\documentclass[a4paper,11pt]{article} 
\usepackage[T1]{fontenc} 
\usepackage{amsmath,amsthm}
\usepackage[francais,english]{babel}
\usepackage[dvips]{graphicx}
\usepackage{todonotes}
\usepackage{color}
\usepackage{soul}
\usepackage{amsfonts,amssymb}
\usepackage{geometry}
\usepackage{hyperref}
\usepackage{stmaryrd}
\usepackage{fancyhdr}
\usepackage{url}
\usepackage{dsfont}
 \usepackage[latin1]{inputenc} 
\theoremstyle{definition}
\newtheorem*{conj*}{Conjecture}

\newtheorem*{ack}{Acknowledgements}
\newtheorem*{thm*}{Theorem}
\newtheorem*{thm1}{Theorem~\ref{compat1}}
\newtheorem*{thm3}{Theorem~\ref{GCDtheorem}}
\newtheorem*{thm2}{Theorem~\ref{banal-part-factors}}
\newtheorem*{thm4}{Theorem~\ref{cuspidal}}

\newtheorem{prop}{Proposition}[section]

\newtheorem{LM}[prop]{Lemma}
\newtheorem{thm}[prop]{Theorem}
\newtheorem{df}[prop]{Definition}
\newtheorem{cor}[prop]{Corollary}

\newtheoremstyle{pourlesremarques}{\topsep}{\topsep}{\normalfont}{}{\bfseries}{.}{ }{}
\theoremstyle{pourlesremarques}
\newtheorem{rem}[prop]{Remark}
\newtheorem*{rem*}{Remark}
\newtheoremstyle{pourlesexemples}{\topsep}{\topsep}{\normalfont}{}{\bfseries}{.}{ }{}
\theoremstyle{pourlesexemples}

\def\presuper#1#2%
  {\mathop{}%
   \mathopen{\vphantom{#2}}^{#1}%
   \kern-\scriptspace%
   #2}

\renewcommand{\o}{\mathfrak{o}}
\newcommand{\p}{\mathfrak{p}}
\newcommand{\w}{\varpi}
\renewcommand{\d}{\delta}

\renewcommand{\l}{\lambda}
\newcommand{\C}{\mathbb{C}}

\newcommand{\Q}{\mathbb{Q}}
\newcommand{\N}{\mathbb{N}}
\newcommand{\Z}{\mathbb{Z}}
\newcommand{\1}{\mathbf{1}}
\newcommand{\sm}{\mathcal{C}^\infty}

\def\GL{\operatorname{GL}}
\def\Hom{\operatorname{Hom}}

\def\Id{\operatorname{Id}}

\def\bJ{\mathbf{J}}

\def\St{\operatorname{St}}
\def\Ql{\overline{\mathbb{Q}_{\ell}}}

\def\Zl{\overline{\mathbb{Z}_{\ell}}}
\def\Fl{\overline{\mathbb{F}_{\ell}}}
\def\nutwo{\nu}
\def\Gal{\operatorname{Gal}}

\def\Mat{\operatorname{Mat}}
\def\ind{\operatorname{ind}}
\def\L{\mathcal{L}}

\def\diag{\operatorname{diag}}

\def\GCD{\mathop{\mathbf{GCD}}}

\def\Ind{\operatorname{Ind}}
\newcommand{\D}{\Delta}

\def\presuper#1#2%
  {\mathop{}%
   \mathopen{\vphantom{#2}}^{#1}%
   \kern-\scriptspace%
   #2}
\title {\textbf{Rankin--Selberg local factors modulo~$\ell$}}
\author{Robert Kurinczuk\footnote{Robert Kurinczuk, Heilbronn Institute for Mathematical Research, Department of Mathematics, University of Bristol, BS8 1TW, United Kingdom. Email: robkurinczuk@gmail.com}, Nadir Matringe\footnote{Nadir Matringe, Universit\'e de Poitiers, Laboratoire de Math\'ematiques et Applications,
T\'el\'eport 2 - BP 30179, Boulevard Marie et Pierre Curie, 86962, Futuroscope Chasseneuil Cedex. Email: Nadir.Matringe@math.univ-poitiers.fr}}

\begin{document}
\maketitle
\begin{abstract}
\noindent 
After extending the theory of Rankin--Selberg local factors to pairs of~$\ell$-modular representations of Whittaker type, of general linear groups over a non-archimedean local field, we study the reduction modulo~$\ell$ of~$\ell$-adic local factors and their relation to these~$\ell$-modular local factors.  While the~$\ell$-modular local~$\gamma$-factor we associate to such a pair turns out to always coincide with the reduction modulo~$\ell$ of the~$\ell$-adic~$\gamma$-factor of any Whittaker lifts of this pair, the local~$L$-factor exhibits a more interesting behaviour; always dividing the reduction modulo-$\ell$ of the~$\ell$-adic~$L$-factor of any Whittaker lifts, but with the possibility of a strict division occurring. In our main results, we completely describe~$\ell$-modular~$L$-factors in the generic case.  We obtain two simple to state nice formulae:  Let~$\pi,\pi'$ be generic~$\ell$-modular representations; then, writing~$\pi_b,\pi'_b$ for their banal parts, we have 
\[L(X,\pi,\pi')=L(X,\pi_b,\pi_b').\]  
Using this formula, we obtain the inductivity relations for local factors of generic representations.  Secondly, we show that
\[L(X,\pi,\pi')=\GCD(r_{\ell}(L(X,\tau,\tau'))),\]
where the divisor is over all integral generic~$\ell$-adic representations~$\tau$ and~$\tau'$ which contain~$\pi$ and~$\pi'$, respectively, as subquotients after reduction modulo~$\ell$.
\end{abstract}
\section{Introduction}
Let~$F$ be a (locally compact) non-archimedean local field of residual characteristic~$p$ and residual cardinality~$q$, and let~$R$ be an algebraically closed field of characteristic~$\ell$ prime to~$p$ or zero.  In this article, following Jacquet--Piatetski-Shapiro--Shalika in \cite{JPS2} for complex representations, we associate local Rankin--Selberg integrals to pairs of~$R$-representations of Whittaker type~$\rho$ and~$\rho'$ of~$\GL_n(F)$ and~$\GL_m(F)$, and show that they define~$L$-factors~$L(X,\rho,\rho')$ and satisfy a functional equation defining local~$\gamma$-factors.  The purpose of this article lies both in the future study of~$R$-representations by these invariants and in the relationship between~$\ell$-modular local factors and the reductions modulo~$\ell$ of~$\ell$-adic local factors (we quote our main theorems towards this goal at the end of this introduction).

The theory of~$\ell$-modular representations of~$\GL_n(F)$ was developed by Vign\'eras in \cite{V}, culminating in her~$\ell$-modular local Langlands correspondence for~$\GL_n(F)$ in \cite{Viginv}, which is characterised initially on supercuspidal~$\ell$-modular representations by compatibility with the~$\ell$-adic local Langlands correspondence. The possibility of characterising such a correspondence with natural invariants forms part of the motivation for this work.  Indeed, already for~$\GL_2(F)$ this is an interesting question, an answer given in this special case by Vign\'eras in \cite{V1}. 

We follow \cite{JPS2} to define local factors for~$R$-representations of Whittaker type, after understanding the splitting of~$R$-Haar measures required for this approach in Section \ref{haarmeasures}.  In Section \ref{SectionRSdefs}, we show that an~$L$-factor attached to~$\ell$-adic representations of Whittaker type is equal to the inverse of a polynomial with coefficients in~$\Zl$, allowing for a natural reduction modulo~$\ell$ map on the set of~$\ell$-adic~$L$-factors. Furthermore, for~$\ell$-modular representations~$\pi$ and~$\pi'$ of Whittaker type of~$\GL_n(F)$ and~$\GL_m(F)$, there exist~$\ell$-adic representations~$\tau$ and~$\tau'$ of Whittaker type of~$\GL_n(F)$ and~$\GL_m(F)$ which stabilise natural~$\Zl$-lattices~$\Lambda$ and~$\Lambda'$ in their respective Whittaker models such that the~$\ell$-modular representations induced on~$\Lambda\otimes_{\Zl}\Fl$ and~$\Lambda'\otimes_{\Zl}\Fl$ are isomorphic to the Whittaker models of~$\pi$ and~$\pi'$. Our first main result is a comparison between the~$L$-factors and local~$\gamma$-factors defined by these two reduction modulo~$\ell$ maps.
%
%
%
%
\begin{thm1}
\begin{enumerate}
\item The~$\ell$-modular~$L$-factor~$L(X,\pi,\pi')$ divides the reduction modulo~$\ell$ of the~$\ell$-adic~$L$-factor~$L(X,\tau,\tau')$.  Moreover, the division of~$L$-factors may not be an equality.  
\item Let~$\theta$ be an~$\ell$-adic character of~$F$. The local~$\gamma$-factor associated to~$\pi,\pi'$ and to the reduction modulo~$\ell$ of~$\theta$ is equal to the reduction modulo~$\ell$ of the local~$\gamma$-factor associated to~$\tau,\tau'$ and~$\theta$.
\end{enumerate}
\end{thm1}

Due to their nice behaviour with reduction modulo~$\ell$, we obtain the inductivity relation of~$\ell$-modular~$\gamma$-factors of representations of Whittaker type (Theorem~\ref{inductivity-gamma}) by choosing appropriate lifts and reducing the~$\ell$-adic inductivity relation of~$\gamma$-factors, and from this we obtain a ``division inductivity relation'' for~$\ell$-modular~$L$-factors (Lemma~\ref{Lintermediate}).  The remainder of the article is concerned with~$\ell$-modular~$L$-factors and their relation to~$\ell$-adic~$L$-factors.  Our first main result in this direction is the complete answer for cuspidal~$\ell$-modular representations.

\begin{thm4}
Let~$\pi_1$ and~$\pi_2$ be two cuspidal~$\ell$-modular representations of~$G_{n_1}$ and~$G_{n_2}$. Then 
$L(X,\pi_1,\pi_2)$ is equal to~$1$, except in the following case:~$\pi_1$ is banal (hence supercuspidal), and~$\pi_2\simeq \chi\pi_1^\vee$ for some 
unramified character~$\chi$ of~$F^\times~$ (in particular~$n_1=n_2$). In this case, let~$e$ be the common ramification index of~$\pi_1$ and~$\pi_2$, we have 
\[L(X,\pi_1,\pi_2)=\frac{1}{1-(\chi(\w_F)X)^{n/e}}\]
and this factor is the reduction modulo~$\ell$ of the~$L$-factor of any cuspidal lifts of~$\pi_1$ and~$\pi_2$.
\end{thm4}

The proof of this theorem is diverse and uses the main result of \cite{KM15} on test vectors for~$\ell$-adic representations for the banal case, the division inductivity relation of~$L$-factors for the cuspidal non-supercuspidal case, and a separate examination of the non-banal supercuspidal case by studying their lifts.

By restricting to pairs of banal generic representations we obtain the inductivity relation of~$L$-factors in this setting, and an explicit formula analogous to the~$\ell$-adic case (Theorem \ref{inductivity-banal-generic}).  Banal~$\ell$-modular representations were introduced and studied for inner forms of~$\GL_n$ by~M\'inguez and S\'echerre in \cite{MSComp}.  In our Preliminaries (Section \ref{repwhit}), we show that a generic representation~$\pi$ of~$\GL_n(F)$ can be written as a product of its nicely behaved banal part~$\pi_b$ and its totally non-banal part~$\pi_{tnb}$.  Our third main result shows that for generic representations the~$L$-factor depends only on the banal parts of the representations. 

\begin{thm2}
Let~$\pi=\pi_{b}\times \pi_{tnb}$ be a generic~$\ell$-modular representations of~$\GL_n(F)$, and~$\pi'=\pi_{b}'\times \pi_{tnb}'$ be a generic~$\ell$-modular representation of~$\GL_{m}(F)$, then we have
\[ L(X,\pi,\pi')= L(X,\pi_b,\pi_b').\]
\end{thm2}
From this theorem and the banal inductivity relation, we obtain the inductivity relation of~$L$-factors for generic~$\ell$-modular representations (Corollary \ref{generic-inductivity}).

Our fourth main result gives an equality between the local~$L$-factor of a pair of generic~$\ell$-modular representations we have defined via~$\ell$-modular Rankin--Selberg integrals, and the greatest common divisor of the reductions modulo~$\ell$ of certain~$\ell$-adic~$L$-factors.
\begin{thm3}
Let~$\pi$ and~$\pi'$ be two generic~$\ell$-modular representations of~$\GL_n(F)$ and~$\GL_m(F)$, then
\[L(X,\pi,\pi')=\GCD(L(X,\tau,\tau')).\]
where the divisor is over all integral generic~$\ell$-adic representations~$\tau$ of~$\GL_n(F)$ and~$\tau'$ of~$\GL_m(F)$ which contain~$\pi$ and~$\pi'$, respectively, as subquotients after reduction modulo~$\ell$.
\end{thm3}

At the end of the article we remark how the work of Vign\'eras together with Theorem \ref{GCDtheorem} suggests a possible generalisation of~$L$-factors to pairs of irreducible~$\ell$-modular representations.  

%
This work further develops the theory of~$\ell$-modular local~$L$-factors of M\'inguez in \cite{M}. In particular, we use his results on Tate~$L$-factors modulo~$\ell$. Recently, using our results from Section \ref{haarmeasures} on the splitting of Haar measures, Moss in \cite{Moss3} has studied Rankin--Selberg~$\gamma$-factors in families over more general rings.  
%

\begin{ack}
We thank Alberto M\'inguez and Vincent S\'echerre for useful explanations.  This work was started at the University of Bristol and finished at the Universit\'e de Poitiers, and the authors would like to thank them for their hospitality. Both authors benefited from support from the grant ANR-13-BS01-0012 FERPLAY.
\end{ack}

\section{Preliminaries}
Before embarking on the study of local~$L$-factors in positive characteristic, we introduce background on representations of the general linear group and extend the general theory.  In particular,  starting with results given in the standard reference \cite{V}, we show how integration behaves with respect to group decompositions.  Indeed, this deserves checking as not all formulae follow from mimicking the proofs in the characteristic zero setting, due to the presence of compact open subgroups of measure zero.   Additionally, in Section \ref{repwhit}, we review the theory of~$\ell$-adic and~$\ell$-modular representations of Whittaker type and reduction modulo~$\ell$, drawing on results originally in \cite{V3}, while our exposition will be influenced by the recent generalisation to inner forms of general linear groups in \cite{MS}.  In Theorem \ref{integral-whittaker}, we prove a technical result on lattices in representations of Whittaker type, in preparation for our results in the next section on the reduction modulo~$\ell$ of~$\ell$-adic local factors. In Section \ref{genericsection}, we specialise to generic representations, and, in particular, notice that a generic~$\ell$-modular representation has a banal part and a totally non-banal part, and prove the commutation of derivatives with reduction modulo~$\ell$.  Finally, in Section \ref{liftingsection}, we show that a non-banal supercuspidal~$\ell$-modular representation has two lifts which are in different inertial classes.

\subsection{Notations}
Let~$F$ be a (locally compact) non-archimedean local field of residual characteristic~$p$ with normalised absolute value~$|~|$.  Let~$\o$ denote the ring of integers in~$F$,~$\p=\w\o$ the unique maximal ideal of~$\o$, and~$q$ the cardinality of~$k=\o/\p$. 

Let~$R$ be a commutative ring with identity of characteristic~$\ell$ not equal to~$p$.  If~$R$ contains a square root of~$q$, we fix such a choice~$q^{1/2}$.

Let
\[\mathcal{M}_{n,m}=\Mat(n,m,F),\quad\mathcal{M}_n=\Mat(n,n,F),\quad G_n=\GL_n(F),\]
and~$\eta$ be the row vector~$(0,\ldots,0,1) \in\mathcal{M}_{1,n}$.  We write~$\nu_R$ for the character~$|~|\circ \det:G_n\rightarrow R^\times$, although we will often more simply write~$\nu$ for~$\nu_R$ when the coefficient ring is clear. For~$k\in\mathbb{Z}$, let
\[G_n^k=\{g\in G_n: \nu_{\mathbb{C}}(g)=q^{-k}\},\] 
and more generally~$X^k=X\cap G_n^k$, for any~$X\subset G_n$.  Let~$B_n$ be the Borel subgroup of upper triangular matrices,~$A_n$ the diagonal torus,~$N_n$ the unipotent radical of~$B_n$. 

We fix a nontrivial smooth character~$\theta$ from~$(F,+)$ to~$R^\times$ and, by abuse of notation, we will denote by~$\theta$ the character~$x\mapsto \theta(\sum_{i=1}^{n-1} n_{i,i+1})$ of~$N_n$.

If~$\lambda$ is an ordered partition of~$n$, we denote by~$P_{\lambda}$ the standard 
parabolic subgroup of~$G_n$ attached to it, by~$M_{\lambda}$ the standard Levi factor of~$P_\lambda$, and by~$N_{\lambda}$ its unipotent radical.
 If~$t+r=n$, we let 
\[U_{t,r}=\left\{\begin{pmatrix} I_t & x \\ & y \end{pmatrix}:x\in \mathcal{M}_{t,r},y\in N_r\right\},\] 
and~$H_{t,r}= G_tU_{t,r}$. By restriction,~$\theta$ defines a character of~$U_{t,r}$. We let~$P_n=H_{n-1,1}$ denote the \emph{mirabolic subgroup} of~$G_n$.

We denote by~$w_n$ the antidiagonal matrix of~$G_n$ with ones on the second diagonal, and if~$n=r+t$, we denote by~$w_{t,r}$ the matrix~$\diag(I_t,w_r)$. Notice that our notations are different from those of  
\cite{JPS1} for~$U_{r,t}$,~$H_{t,r}$, and~$w_{t,r}$. 

If~$\Phi \in\mathcal{C}_c^\infty(F^n)$, we denote by~$\widehat{\Phi}$ its Fourier transform with respect to the~$\theta$-self-dual~$R$-Haar measure~$dx$ on~$F^n$ satisfying~$dx(\o)=q^{-l/2}$, where the integer~$l$ satisfies that~$\theta\mid_{\p^l}$ is trivial, but~$\theta\mid_{\p^{l-1}}$ is non-trivial.

For~$G$ a locally profinite group, we let~$\mathfrak{R}_R(G)$ denote the abelian category of smooth~$R$-representations of~$G$.  All~$R$-representations henceforth considered are assumed to be smooth.  For~$\pi$ an~$R$-representation with central character, for example an~$R$-representation of~$G_n$ parabolically induced from an irreducible~$R$-representation, we denote its central character by~$c_{\pi}$.

Let~$\Ql$ be an algebraic closure of the~$\ell$-adic numbers,~$\Zl$ its ring of integers, and~$\Fl$ its residue field, which is an algebraic closure of the finite field of~$\ell$ elements. By an \emph{$\ell$-adic representation} of~$G$ we mean a representation of~$G$ on a~$\Ql$-vector space, and by an \emph{$\ell$-modular} representation of~$G$ we mean a representation of~$G$ on a~$\Fl$-vector space.  For~$H$ a closed subgroup of~$G$, we write~$\Ind_H^G$ for the functor of normalised smooth induction from~$\mathfrak{R}_R(H)$ to~$\mathfrak{R}_R(G)$, and write~$\ind_H^G$ for the functor of normalised smooth induction with compact support. 
 
We assume that our choice of square roots of~$q$ in~$\Fl$ and~$\Ql$ are \emph{compatible}; in the sense that the former is the reduction modulo~$\ell$ of the latter, which is chosen in~$\Zl$.  

\subsection{$R$-Haar Measures}\label{haarmeasures}
Let~$R$ be a commutative ring with identity of characteristic~$\ell$ and let~$G$ be a locally profinite group which admits a compact open subgroup of pro-order invertible in~$R$. We let 
\[\mathcal{C}_c^{\infty}(G,R)=\{f:G\rightarrow R:f \text{ is locally constant and compactly supported}\},\] 
(we sometimes write this, more simply, as~$\mathcal{C}_c^{\infty}(G)$ according to the context). A \emph{left} (resp. \emph{right}) \emph{~$R$-Haar measure} on~$G$ is a nonzero linear form on~$\mathcal{C}_c^{\infty}(G,R)$ which is invariant under left (resp. right) translation by~$G$.  If~$\mu$ is a left (or right)~$R$-Haar measure on~$G$ and~$f\in \mathcal{C}_c^{\infty}(G,R)$, we write
\[\mu(f)=\int_G f(g)\mu.\]
By \cite[I 2.4]{V}, for each compact open subgroup~$K$ of~$G$ of pro-order invertible in~$R$ there exists a unique left~$R$-Haar measure~$\mu$ such that~$\mu(K)=1$. The \emph{volume}~$\mu(K')=\mu(\mathbf{1}_{K'})$ of a compact open subgroup~$K'$ of~$G$ is equal to zero if and only if 
the pro-order of~$K'$ is equal to zero in~$R$. In the present work, the \emph{modulus character} of~$G$ is the unique character~$\delta_G:G\rightarrow R^\times$ such that, if~$\mu$ is a left~$R$-Haar measure on~$G$,~$\delta_G\mu$ is a right~$R$-Haar measure on~$G$. 
More generally, if~$H$ is a closed subgroup of~$G$, we let~$\delta={\delta_{G}^{-1}}\mid_{H}\delta_H$, and  
\[\mathcal{C}_c^{\infty}(H\backslash G,\delta, R)\] be the space of functions from~$G$ to~$R$, fixed on the right by a compact open subgroup of~$G$, compactly supported modulo~$H$, and
which transform by~$\delta$ under H on the left  (we sometimes write this as~$\mathcal{C}_c^{\infty}(H\backslash G,\delta)$). For~$f\in \mathcal{C}_c^{\infty}(G,R)$, we denote by~$f^H$ the function in 
$\mathcal{C}_c^{\infty}(H\backslash G,\delta, R)$ defined by 
\[f^H(g)=\int_{H} f(h)\d^{-1}(h)dh,\] 
for~$dh$ a right~$R$-Haar measure on~$H$. 
It is proved in \cite[I 2.8]{V} that the map~$f\mapsto f^H$ is surjective, and that there is a unique, up to an invertible scalar, nonzero linear form~$d_{H\backslash G}g$ on~$\mathcal{C}_c^{\infty}(H\backslash G,\delta, R)$, which is right invariant under~$G$. We call such a nonzero linear form on~$\mathcal{C}_c^{\infty}(H\backslash G,\delta, R)$ a \emph{$\delta$-quasi-invariant quotient measure} on~$H\backslash G$ and, for~$f\in\mathcal{C}_c^{\infty}(H\backslash G,\delta, R)$, we write
\[d_{H\backslash G}g(f)=\int_{H\backslash G} f(g)d_{H\backslash G}g.\]
With the correct normalisation, for all~$f\in\mathcal{C}_c^{\infty}(G,R)$, we have 
\[\int_{G}f(g)dg=\int_{H\backslash G} f^H(g)d_{H\backslash G}g.\]

\begin{rem}\label{intvalues}
Let~$H$ be a closed subgroup of~$G$.  Let~$\mathcal{C}_{c,e}^\infty(H)$ denote the subspace of functions in~$\mathcal{C}_c^\infty(H,\Ql)$ which take integral values. Up to a correct normalisation of the~$\Ql$-Haar measure~$dh$ on~$H$, for all 
$f\in \mathcal{C}_{c,e}^\infty(H)$, the integral~$\int_{H} f(h)dh$ belongs to~$\Zl$.  Suppose~$K$ is a closed subgroup of~$H$, for which there is a~$\delta$-quasi-invariant quotient measure~$d_{K\backslash H}h$ on~$K\backslash H$.  We write~$\mathcal{C}_{c,e}^\infty(K\backslash H,\delta)$ for the subspace of functions in~$\mathcal{C}_{c}^\infty(K\backslash H,\delta)$ which take integral values.  Similarly, up to correct normalisation of the quotient measure, the value of 
$\int_{K \backslash H} f(h)d_{K\backslash H}h$ belongs to~$\Zl$ when~$f\in \mathcal{C}_{c,e}^\infty(K\backslash H,\delta)$.  Moreover, for all 
$f$ in~$\mathcal{C}_{c,e}^\infty(K\backslash H,\delta)$, we have \[r_{\ell}\left(\int_{K \backslash H} f(h)d_{K\backslash H}h\right)=\int_{K \backslash H} r_{\ell}(f(h))d_{K\backslash H}h.\] 
We write~$\mathcal{C}_{c,e}^\infty(F^n)$ for the subspace of functions in~$\mathcal{C}_c^\infty(F^n)$ which take values in~$\Zl$.  Note that if~$\Phi\in \mathcal{C}_{c,e}^\infty(F^n)$, then~$\hat{\Phi}\in \mathcal{C}_{c,e}^\infty(F^n)$.
\end{rem}

For the remainder of this section, let~$G$ denote a unimodular locally profinite group. Suppose that~$B$ is a closed subgroup of~$G$,~$K$ is a compact open subgroup 
of~$G$ such that~$G=BK$.

\begin{LM}\label{Haar1}
Let~$dg$ be an~$R$-Haar measure on~$G$.  There exist a right~$R$-Haar measure~$db$ on~$B$ and a right~$K$-invariant measure~$dk$ on~$K\cap B\backslash K$ such that, for all~$f\in\sm_c(G,R)$, we have
\[\int_{G}f(g)dg= \int_{K\cap B\backslash K} \int_{B}f(bk)\delta_B(b)^{-1}db dk.\] 
\end{LM}
\begin{proof}
We observe first that the map~$\phi\mapsto \phi\mid_{K}$ is a vector space isomorphism between~$\mathcal{C}_c^{\infty}(B\backslash G,\delta_{B}, R)$ and
$\mathcal{C}_c^{\infty}((K\cap B)\backslash K, R).$ It is injective because~$G=BK$. To show surjectivity, we recall that the characteristic functions~$\1_{(K\cap B)k \mathcal{U}}$, with~$ \mathcal{U}$ a compact subgroup of~$K$ of 
pro-order invertible in~$R$ and~$k\in K$,  span~$\mathcal{C}_c^{\infty}((K\cap B)\backslash K, R)$. Moreover, the function~$\mathbf{1}_{k \mathcal{U}}^B$ belongs to~$\mathcal{C}_c^{\infty}(B\backslash G,\delta_{B}, R)$, 
and a computation shows~${\mathbf{1}_{k \mathcal{U}}^B}\mid_{K}=db(B\cap k \mathcal{U}k^{-1})\mathbf{1}_{(K\cap B)k \mathcal{U}}$. Surjectivity follows as~$db(B\cap k \mathcal{U}k^{-1})$ is invertible in~$R$. In particular, if~$dk$ is a right~$K$-invariant measure on~$\mathcal{C}_c^{\infty}((K\cap B)\backslash K, R)$, the map~$\mu:\phi\mapsto dk(\phi\mid_{K})$ is a right~$G$-invariant measure on 
$B\backslash G$. The result follows from the formula~$\int_{G}f(g)dg=\mu(f^B)$.
\end{proof}

\begin{rem}
Let~$K_n=\GL_n(\mathfrak{o}_F)$.  By the Iwasawa decomposition, we have~$G_n=B_nK_n$. Let~$\mu_{G_n}$ be an~$R$-Haar measure on~$G_n$. If~$\ell=0$, or more generally~$\ell\nmid q-1$, then, for all~$f\in\mathcal{C}_c^\infty(G_n,R)$, we have
\[\int_{G_n}f(g)dg=\int_{B_n}\int_{K_n}f(bk)dbdk,\]
for good choices of a left~$R$-Haar measure~$db$ on~$B_n$ and an~$R$-Haar measure~$dk$ on~$K_n$. As noticed by M\'inguez in \cite{M}, this is no longer true in general. More 
precisely, it is not true when~$\ell\mid q-1$ as the restriction of an~$R$-Haar measure on~$K$ to~$\mathcal{C}_c^\infty((K_n\cap B_n)\backslash K_n)$ is zero. That is why we use a right 
invariant measure on~$K\cap B\backslash K$ in Lemma \ref{Haar1}.
\end{rem}

%
%

Let~$Q$,~$L$ and~$N$ be closed subgroups of~$G$ such that~$Q=LN$ and~$L$ normalises~$N$.  Suppose that there exists a compact open subgroup~$\mathcal{U}$ of~$G$ of pro-order invertible in~$R$ such that~$Q\cap \mathcal{U}=(N\cap \mathcal{U})(M\cap \mathcal{U})$. Let~$dl$ be an~$R$-Haar measure on~$L$ and~$dn$ be an~$R$-Haar measure on~$N$. 

\begin{LM}\label{HaarB}
Let~$f\in\mathcal{C}^\infty_c(Q,R)$.  There exists a unique right~$R$-Haar measure~$dq$ on~$Q$ such that 
\[\int_Qf(q)dq=\int_L\int_Nf(nl)dldn.\]
\end{LM}

\begin{proof}
As~$L$ normalises~$N$, we see that~$\int_L\int_Nf(nl)dldn$ is a right~$Q$-invariant linear form on~$C_c^\infty(Q,R)$. But as~$Q\cap  \mathcal{U}=(N\cap  \mathcal{U})(M\cap  \mathcal{U})$, it is easy to see that~$\int_L\int_N1_{Q\cap  \mathcal{U}}(nl)dldn$ is nonzero. 
\end{proof}

\begin{rem}
A typical instance is when~$G=G_n$,~$Q=LN$ is a standard parabolic subgroup of~$G_n$, and~$K_1$ is the pro-$p$ unipotent radical of~$K_n$.
\end{rem}

We have the following corollary to Lemmas \ref{Haar1} and \ref{HaarB}.

\begin{cor}\label{iwasplit}
Let~$dg$ be an~$R$-Haar measure on~$G$.  There exist an~$R$-Haar measure~$da$ on~$A_n$, and a right~$K_n$-invariant measure~$dk$ on~$(K_n\cap B_n)\backslash K_n$ such that, for all~$f\in \mathcal{C}_c^\infty(N_n\backslash G_n,R)$, we have
\[\int_{N_n\backslash G_n}f(g)dg= \int_{(K_n\cap B_n)\backslash K_n}\int_{A_n} f(ak)\d_{B_n}(a)^{-1}dadk.\]
\end{cor}
\begin{proof} 
Let~$f\in \mathcal{C}_c^\infty(N_n\backslash G_n,R)$, then~$f=h^{N_n}$ for 
$h\in \mathcal{C}_c^\infty(G_n,R)$. By Lemmas \ref{Haar1} and \ref{HaarB},
\[\int_{G_n} h(g)dg=\int_{K_n\cap B_n\backslash K_n}\int_{ A_n } \int_{N_n}h(nak)\d_{B_n}(a)^{-1}dn da dk,\] 
However, \[\int_{G_n} h(g)dg=\int_{N_n\backslash G_n}f(g)dg\] and 
\[\int_{N_n} h(nak)dn=h^{N_n}(ak)=f(ak).\] 
\end{proof}

From Iwasawa decomposition, we also have~$G_n=P_nZ_nK_n$. We use the following integration formula, which is proved in a similar fashion.

\begin{cor}\label{iwasplit2}
Let~$dg$ be an~$R$-Haar measure on~$G$.  There exist an~$R$-Haar measure~$dz$ on~$Z_n$, a~$\delta$-quasi-invariant quotient measure~$dp$ on~$N_n\backslash P_n$, and a right~$K_n$-invariant measure~$dk$ on~$(K_n\cap B_n)\backslash K_n$ such that, for all~$f\in \mathcal{C}_c^\infty(N_n\backslash G_n,R)$,
\[\int_{N_n\backslash G_n}f(g)dg= \int_{(K_n\cap P_n)\backslash K_n}\int_{Z_n}\int_{N_n\backslash P_n}  f(pzk)\det(p)^{-1}dpdzdk.\]
\end{cor}

Henceforth, equalities involving integrals will be true only up to the correct normalisation of measures.

\subsection{Derivatives}\label{der}

Henceforth, we suppose that~$R$ is an algebraically closed field.  Following \cite{BZ}, we define the following exact functors: 
\begin{enumerate}
\item~$\Psi^+:\mathfrak{R}_R(G_{n-1})\rightarrow \mathfrak{R}_R(P_n)$, extension by the trivial representation twisted by~$\nu^{\frac{1}{2}}$.
\item~$\Psi^-: \mathfrak{R}_R(P_n)\rightarrow\mathfrak{R}_R(G_{n-1})$, the functor of~$U_{n-1}$-coinvariants twisted by~$\nu^{-\frac{1}{2}}$.
\item~$\Phi^+:\mathfrak{R}_R(P_{n-1})\rightarrow \mathfrak{R}_R(P_n)$, the functor~$\Phi^+(X)=\ind_{P_{n-1}U_n}^{P_n}(X\otimes\theta)$.
\item~$\Phi^-: \mathfrak{R}_R(P_{n})\rightarrow\mathfrak{R}_R(P_{n-1})$, the functor of~$(U_{n-1},\theta)$-coinvariants twisted by~$\nu^{-\frac{1}{2}}$.
\item~$\Phi^+_{nc}:\mathfrak{R}_R(P_{n-1})\rightarrow \mathfrak{R}_R(P_n)$, the functor~$\Phi^+_{nc}(X)=\Ind_{P_{n-1}U_n}^{P_n}(X\otimes\theta)$.
\end{enumerate}

\begin{thm}[{\cite{BZ0} \& \cite{BZ} (cf. \cite[III 1.3]{V})}]
~
\begin{enumerate}
\item We have~$\Psi^{-}\Phi^+=\Psi^-\Phi_{nc}^+=\Phi^-\Psi^+=0$, and~$\Psi^-$, resp.~$\Phi^+$, resp.~$\Phi^-$, is left adjoint to~$\Psi^+$, resp.~$\Phi^-$, resp.~$\Phi^+_{nc}$.
\item The identity functor~$\1:\mathfrak{R}_R(P_n)\rightarrow \mathfrak{R}_R(P_n)$ admits a filtration
\[0=T_n\subset T_{n-1}\subset \cdots\subset T_0=\1\]
such that~$T_k=(\Phi^+)^k(\Phi^-)^k$ and~$T_{k-1}/T_k=(\Phi^+)^{k-1}\Psi^+\Psi^-(\Phi^-)^{k-1}$.
\end{enumerate}
\end{thm}

Let~$\tau$ be an~$R$-representation of~$P_n$.  Let~$\tau_{(k)}=(\Phi^-)^{k-1}(\tau)$.  The \emph{$k$-th derivative}~$\tau^{(k)}$ of~$\tau$ is defined by~$\tau^{(k)}=\Psi^-(\tau_{(k)})$.  We recall the classification of irreducible~$R$-representations of~$P_n$.
\begin{LM}[{\cite{BZ0} (cf. \cite[III 1.5]{V})}]
Let~$\tau$ be an irreducible~$R$-representation of~$P_n$.  There exists a unique nonzero derivative of~$\tau$.  Furthermore, for k in~$\{1,\ldots,n\}$, if~$\tau^{(k)}\neq 0$ then~$\tau=(\Phi^+)^k\Psi^+(\tau^{(k)})$.  Conversely, if~$\rho$ is an irreducible~$R$-representation of~$G_{n-k}$, then~$\pi=(\Phi^+)^k\Psi^+(\rho)$ is irreducible and~$\rho$ is the~$k$-th derivative of~$\pi$.
\end{LM}

Let~$\pi$ be an~$R$-representation of~$G_n$.  The \emph{zeroth derivative}~$\pi^{(0)}$ of~$\pi$ is~$\pi$.  Let~$\tau=\pi\mid_{P_n}$ and set~$\pi_{(k)}=\tau_{(k)}$ for~$k=0,1,\ldots,n$. Define the \emph{$k$-th derivative}~$\pi^{(k)}$ of~$\pi$ by~$\pi^{(k)}=\tau^{(k)}$, for~$k=1,\ldots,n$.

\begin{LM}[{\cite{BZ0}}]\label{dimensionwhittaker}
 Let~$\pi$ be an~$R$-representation of finite length. Then the dimension of~$\pi^{(n)}$ is finite and equal to the dimension of~$\Hom_{N_n}(\pi,\theta)$. 
\end{LM}

We remark once and for all that the functor
\[\pi\mapsto \Hom_{N_n}(\pi,\theta)\] 
is exact from~$\mathfrak{R}_R(G_n)$ to the category of~$R$-vector spaces, because the functor~$\pi\mapsto \pi^{(n)}$ is exact, and will use this fact a lot without mentioning it. We will also use several times the following proposition.

\begin{prop}[{\cite[Proposition 3.7]{BZ}}]\label{BZ3.7}
Let~$\rho$ and~$\rho'$ be~$R$-representations of~$G_n$ and~$\tau$ and~$\tau'$ be~$R$-representations of~$P_m$, we have
\begin{enumerate}
\item~$\Hom_{P_{n+1}}(\Psi^+(\rho)\otimes \Psi^+(\rho'),R)=\Hom_{G_n}(\rho\otimes \rho',R)$;
\item~$\Hom_{G_n}(\Phi^+(\tau)\otimes \Phi^+(\tau'),R)=\Hom_{P_n}(\tau\otimes \tau',R)$;
\item~$\Hom_{G_n}(\Psi^+(\rho)\otimes \Phi^+(\tau),R)=\{0\}$.
\end{enumerate}
\end{prop}
%
The derivatives of a product are given by the Leibniz rule.

\begin{LM}[{\cite{BZ} (cf. \cite[III 1.10]{V})}]\label{derprod}
Suppose~$\pi$ is an~$R$-representation of~$G_n$ and~$\rho$ is an~$R$-representation of~$G_m$, then~$(\pi\times\rho)^{(k)}$ has a filtration with successive quotients~$\pi^{(i)}\times\rho^{(k-i)}$, for~$0\leqslant i\leqslant k$.
\end{LM}

It is ``well known to experts'' that the derivative functors commute with reduction modulo~$\ell$. However, 
it seems that no proof appears in the literature. We will use this property in the sequel, but postpone its proof to Subsection \ref{genericsection}. 

\begin{thm}\label{derivatives-commute}
Let~$\tau$ be an integral~$\ell$-adic representation of finite length of~$G_n$, then for~$0\leq k\leq n$, one has 
$r_\ell(\tau^{(k)})=[r_\ell(\tau)^{(k)}]$ (where the square brackets stand for the semi-simplification again).
\end{thm}
\begin{proof}
See just after Corollary \ref{reduction-commutes-with-derivatives-2}.
\end{proof}

\subsection{Parabolic induction and restriction, integral structures and reduction modulo~$\ell$}\label{sectintegral}

Let~$Q$ be a parabolic subgroup of~$G_n$ with Levi factor~$M$.  We write~$\boldsymbol{i}^{G_n}_Q$ for the functor of normalised parabolic induction from~$\mathfrak{R}_R(M)$ to~$\mathfrak{R}_R(G_n)$. If~$\tau=\pi_1\otimes   \dots \otimes \pi_r$ is a smooth~$R$-representation of~$M_{(m_1,\dots,m_r)}$ with~$\sum_i m_i=n$, we will use the product notation~$\pi_1\times\dots \times \pi_r$ for the induced representation~$\boldsymbol{i}^{G_n}_{P_{(m_1,\dots,m_r)}}(\tau)$.  An irreducible~$R$-representation of~$G_n$ is called \emph{cuspidal} if it is irreducible and it does not appear as a subrepresentation of any parabolically induced representation. It is called \emph{supercuspidal}, moreover, if it does not appear as a subquotient of any parabolically induced representation.

Let~$G$ be a direct product of the form~$G_1\times \dots \times G_r$. An~$\ell$-adic representation~$(\pi,V)$ of~$G$ is called \emph{integral} if it has finite length, and if~$V$ contains a~$G$-stable~$\Zl$-lattice~$\Lambda$. Such a lattice is called an \emph{integral structure}, or \emph{lattice}, in~$\pi$.  A character is integral if and only if it takes values in~$\Zl$.  By \cite[II  4.12]{V}, a cuspidal representation is integral if and only if its central character is integral.

If~$\pi$ is an integral~$\ell$-adic representation with integral structure~$\Lambda$, then~$\pi$ defines an~$\ell$-modular representation on the space~$\Lambda\otimes_{\Zl}\Fl$.  By the Brauer-Nesbitt principle \cite[Theorem 1]{V2}, the semi-simplification, in the Grothendieck group of finite length~$\ell$-modular representations, of~$(\pi,\Lambda\otimes_{\Zl}\Fl)$ is independent of the choice of integral structure in~$\pi$ and we call this semisimple representation~$r_{\ell}(\pi)$ the \emph{reduction modulo~$\ell$} of~$\pi$. We say that an~$\ell$-modular representation~$\pi$ \emph{lifts} to an integral~$\ell$-adic representation~$\tau$ if~$r_{\ell}(\tau)\simeq \pi$, we will only really use this notion of lift when~$\pi$ is irreducible. 

Let~$H$ be a closed subgroup of~$G$,~$\sigma$ be an integral~$\ell$-adic representation of~$H$, and~$\Lambda$ be an integral structure in~$\sigma$.  By \cite[I  9.3]{V},~$\ind_H^G(\Lambda)$ is an integral structure in~$\ind_H^G(\sigma)$.  Suppose~$Q$ is a parabolic subgroup of~$G$ with Levi decomposition~$Q=MN$,~$\sigma$ is an integral representation of~$M$,~$\Lambda$ is an integral structure in~$\sigma$. Then~$\boldsymbol{i}_Q^G(\Lambda)$ is an integral structure in~$\boldsymbol{i}_Q^G(\sigma)$. Moreover, we have~$\boldsymbol{i}_Q^G(\Lambda\otimes_{\Zl}\Fl)\simeq \boldsymbol{i}_Q^G(\Lambda)\otimes_{\Zl}\Fl$. Hence \emph{parabolic induction commutes with reduction modulo~$\ell$}, in the sense that
\[ r_{\ell}(\boldsymbol{i}_Q^G(\sigma))\simeq [\boldsymbol{i}_Q^G(r_{\ell}(\sigma))],\]
where the square brackets indicate we take the semi-simplification of~$\boldsymbol{i}_Q^G(r_{\ell}(\sigma))$. 

If~$\pi$ is an~$R$-representation of~$G$, and~$Q$ is a parabolic subgroup of~$G$, 
then we denote by~$\boldsymbol{r}_Q^G(\pi)$ its normalised \textit{Jacquet module}. In a similar way, if~$\tau$ is an integral~$\ell$-adic representation of~$G$ of finite length, thanks to \cite[Proposition 6.7]{Dat}, the Jacquet module~$\boldsymbol{r}_Q^G(\tau)$ is integral, and one has the isomorphism
\[r_{\ell}(\boldsymbol{r}_Q^G(\tau))\simeq [\boldsymbol{r}_Q^G(r_{\ell}(\tau))].\].

\subsection{Representations of Whittaker type}\label{repwhit}
We now recall facts originally from \cite{V3}, with our exposition and notation following \cite[Section 7]{MS}.  First, recall the following definition:

\begin{df} An irreducible~$R$-representations of~$G_n$ satisfying~$\Hom_{N_n}(\pi,\theta)\neq 0$ is called \textit{generic}.  
\end{df}

A \emph{segment}~$\D=[a,b]_\rho$ is a sequence~$(\nu^a\rho,\nu^{a+1}\rho,\cdots,\nu^b\rho)$ with~$\rho$ a cuspidal~$R$-representation of~$G_m$ for some~$m\geqslant 1$, and~$a,b\in\mathbb{Z}$ with~$a\leqslant b$. Its \emph{length} is, by definition,~$b-a+1$. Two segments 
~$\D=[a,b]_\rho$ and~$\D'=[a',b']_\rho'$ are said to be equivalent if they have the same length, and~$\nu^a\rho\simeq \nu^{a'}\rho'$. 
Hence, as noticed in \cite[7.2]{MS}, the segment~$[a,b]_\rho$ identifies with the cuspidal pair \[r_\D=(M_{(m,\dots,m)},\nu^a\rho\otimes \nu^{a+1}\rho\otimes \cdots \otimes \nu^b\rho),\] and the equivalence relation on segments is the restriction of the classical isomorphism equivalence relation on cuspidal pairs. To such 
a segment~$\D$, in \cite[Definition 7.5]{MS} the authors associate a certain quotient~$\L(\Delta)$ of 
$\nu^a\rho\times \nu^{a+1}\rho\times \cdots \times \nu^b\rho$. The representation~$\L(\D)$ in fact determines~$\D$, as 
its normalised Jacquet module with respect to the opposite of~$N_{(m,\dots,m)}$ is equal to~$r_\Delta$ according to \cite[Lemma 7.14]{MS}. The conclusion of this is that the objects 
$\D$,~$\L(\Delta)$, and~$r_\Delta$ determine one another, and hence we call~$\L(\Delta)$ a segment.  

We say that~$\Delta$ \emph{precedes}~$\Delta'$ if we can extract from the sequence~$(\nu^a\rho, \ldots, \nu^b\rho,\nu^{a'}\rho',\ldots,\nu^{b'} \rho')$ a subsequence which is a segment of length strictly larger than both the length of~$\Delta$ and the length of~$\Delta'$.  We say~$\Delta$ and~$\Delta'$ are \emph{linked} if~$\Delta$ precedes~$\Delta'$ or~$\Delta'$ precedes~$\Delta$.

Let~$\rho$ be a cuspidal~$R$-representation of~$G_m$, we denote by~$\Z_\rho$ the \emph{cuspidal line} of~$\rho$, that is the set 
\[\Z_\rho=\{\nu^r \rho,r\in \Z\},\] 
which when~$\nu$ has finite order in~$R$ is perhaps more appropriately thought of as a circle.  We define an element~$o(\rho)$ in~$\N \cup \{+\infty\}$ by the formula 
\[o(\rho)=|\Z_\rho|.\]
In \cite[5.2]{MS}, a positive integer~$e(\rho)$ is attached to~$\rho$, 
\begin{equation*}
e(\rho)=\begin{cases}
+\infty&\text{ if }R=\Ql;\\
|\{\nu^r \rho,r\in \Z\}|&\text{ if }R=\Fl\text{ and } o(\rho)>1;\\
\ell&\text{ if }R=\Fl\text{ and } o(\rho)=1.
\end{cases}
\end{equation*}
An integer~$f(\rho)$ is also defined in \cite[5.2]{MS} via type theory, the definition of which we will recall later, and following this reference we set~$q(\rho)=q^{f(\rho)}$. If~$o(\rho)=1$, then~$e(\rho)$ is the order of~$q(\rho)$. For integers~$a\leqslant b$, we denote by~$St(\rho,[a,b])$ the \emph{generalised Steinberg representation} associated with~$\D=[a,b]_\rho$, i.e.~the unique generic subquotient of\[\nu^a\rho\times \nu^{a+1}\rho\times \cdots \times \nu^b\rho.\] By \cite[Remarque 8.14]{MS}, the representation~$St(\rho,[a,b])$ is equal to the segment~$\L(\D)$ if and only if its length~$b-a+1<e(\rho)$. In this case, we say that~$\L(\D)$ is a \textit{generic segment}. As in \cite{MS}, we will write~$St(\rho,k)$ for~$St(\rho,[0,k-1])$, for~$k\geqslant 1$.  

By \cite[III 4.25 \& 5.10]{V}, if~$\rho$ is a cuspidal~$\ell$-modular representation of~$G_r$, then it is the reduction modulo~$\ell$ of an integral cuspidal~$\ell$-adic representation~$\sigma$ of~$G_r$, and the reduction modulo-$\ell$ of any cuspidal~$\ell$-adic representation is still cuspidal.

\begin{df}\label{banal-segment}
Let~$\rho$ be an~$R$-cuspidal representation of~$G_n$, and take~$a\leqslant b$ in~$\N$. We say that a segment~$[a,b]_\rho$ is \emph{banal} if its length~$b-a+1<o(\rho)$. In this case, we also say that~$\L(\D)$ is \emph{banal}. 
\end{df}

Let~$\D=[a,b]_\rho$ be a segment. We notice that at once that if the representation~$\L(\D)$ is banal, then it is generic, and that the generic segment~$\L(\D)$ is non-banal if and only if~$o(\rho)=1$, i.e. if and only if~$\rho$ is non-banal in the sense of \cite{MSComp}. 

\begin{rem}\label{lift-segments} Let~$\L(\D)=St(\rho,k)$, then~$\L(\D)$ is the reduction modulo~$\ell$ of~$\L(D)=St(\sigma,k)$. Indeed, 
it is a straightforward consequence (this was already noticed in \cite[V.7.]{V3}) of Theorem \ref{derivatives-commute}.
\end{rem}

\begin{df}\label{Wtype}
An~$R$-representation~$\pi$ of~$G_n$ is called of \emph{Whittaker type} if it is parabolically induced from generic segments, i.e. if~$\pi= \L(\Delta_1)\times \cdots \times \L(\Delta_t)$ with~$\Delta_i$ generic segments, for~$1\leqslant i\leqslant t$.
\end{df}

Let~$\pi$ be a representation of Whittaker type. By Lemmas \ref{dimensionwhittaker} and \ref{derprod}, the space~$\Hom_{N_n}(\pi,\theta)$ is of dimension~$1$, and we denote by~$W(\pi,\theta)$ the \emph{Whittaker model} of~$\pi$, i.e.~$W(\pi,\theta)$ denotes the image of~$\pi$ in~$\Ind_{N_n}^{G_n}(\theta)$. Note that, a representation of Whittaker type may not be irreducible, however, it is of finite length. In fact, thanks to the results of Zelevinsky (cf. {\cite{Z}) in the~$\ell$-adic setting, and by \cite[Theorem 5.7]{V3} (cf. \cite[Theorem 9.10 and Corollary 9.12]{MS} for another proof) in the~$\ell$-modular setting, the irreducible representations of Whittaker type of~$G_n$ are exactly the generic representations. 

If~$\pi$ is a smooth representation of~$G_n$, we denote by~$\widetilde{\pi}$ the representation~$g\mapsto \pi(^t\!g^{-1})$ of 
$G_n$. Let~$\tau$ be an~$\ell$-adic irreducible representation of~$G_n$, then~$\widetilde{\tau}\simeq \tau^\vee$, by 
\cite{GK}. Hence when~$\L(\D)$ is an~$\ell$-modular generic segment of~$G_n$, it lifts to an~$\ell$-adic segment~$\L(D)$ according to Remark \ref{lift-segments}, and as~$\widetilde{\L(D)}\simeq \L(D)^\vee$, we deduce by reduction modulo~$\ell$, cf. \cite[I 9.7]{V}, that~$\widetilde{\L(\D)}\simeq \L(\D)^\vee$. If 
$\pi=\L(\D_1)\times \dots \times \L(\D_t)$ is a representation of~$G_n$ of Whittaker type, we have 
$\widetilde{\pi}=\widetilde{\L(\D_t)}\times \dots \times \widetilde{\L(\D_1)}=
\L(\D_t)^\vee\times \dots \times \L(\D_1)^\vee,$ and we deduce that 
$\widetilde{\pi}$ is also of Whittaker type. In order to state the functional equation for~$L$-factors of representations of Whittaker type, we will need the following lemma whose proof follows from the discussion above and that~$\widetilde{~}$ is an involution.
\begin{LM}
Let~$\pi$ be an~$\ell$-modular representation of Whittaker type of~$G_n$, then~$\widetilde{\pi}$ is of Whittaker type and the map~$W\mapsto \widetilde{W}$, where~$ \widetilde{W}(g)=W(w_n \presuper{t}g^{-1})$, is an~$R$-vector space isomorphism 
between~$W(\pi,\theta)$ and~$W(\widetilde{\pi},\theta^{-1})$.
\end{LM}

For the proofs to come, it will be convenient to choose certain ``lifts'' of~$\ell$-modular representations of Whittaker type. Let us introduce the following notation first.

\begin{df}
Let~$\tau$ be an integral~$\ell$-adic representation of Whittaker type. We denote by~$W_e(\tau,\theta)$ the subset of all functions in~$W(\tau,\theta)$ which take integral values. 
\end{df}

Now we define ``Whittaker lifts'' of representations of Whittaker type. We will soon see that they exist.

\begin{df}\label{Whittaker-lift}
Let~$\pi$ be an~$\ell$-modular representation of Whittaker type. An integral~$\ell$-adic representation of Whittaker type~$\tau$ is a \emph{Whittaker lift} of~$\pi$ if~$W_e(\tau,\theta)$ contains a sublattice~$\mathcal{W}$ such that~$W(\pi,r_{\ell}(\theta))=\mathcal{W}\otimes_{\Zl}\Fl$.  
\end{df}

\begin{rem}
Whittaker lifts are not strictly speaking lifts.  However, if~$\tau$ is a Whittaker lift of~$\pi$ then we can lift functions in the Whittaker model of~$\pi$ to integral functions in the Whittaker model of~$\tau$.
\end{rem}

We will use the following lemma to obtain integrality results on~$\ell$-adic local factors.

\begin{LM}\label{dual-lattice} Let~$\pi$ be an~$\ell$-modular representation of~$G_n$, and let~$\pi$ be a Whittaker lift of 
$\tau$ with associated sublattice~$\mathcal{W}\subset W_e(\tau,\theta)$, then the~$\Zl$-module~$\widetilde{\mathcal{W}}=\{\widetilde{W}:W\in \mathcal{W}\}$ is a sublattice 
of~$W_e(\widetilde{\tau},\theta^{-1})$ such that~$W(\widetilde{\pi},r_{\ell}(\theta))=\widetilde{\mathcal{W}}\otimes_{\Zl}\Fl$.
\end{LM}
\begin{proof}
As~$W\mapsto \widetilde{W}$ is a vector space isomorphism between~$W(\tau,\theta)$ 
and~$W(\widetilde{\tau},\theta^{-1})$ which restricts to a~$\Zl$-module isomorphism between~$W_e(\tau,\theta)$ and~$W_e(\widetilde{\tau},\theta^{-1})$, the set~$\widetilde{\mathcal{W}}$ is a sublattice of~$W_e(\widetilde{\tau},\theta^{-1})$. Moreover, as~$W\mapsto \widetilde{W}$ is also a vector space isomorphism between~$W(\pi,r_\ell(\theta))$ and~$W(\widetilde{\pi},r_\ell(\theta^{-1}))$, and because~$r_\ell(\widetilde{W})=\widetilde{r_\ell(W)}$ for an~$W\in W_e(\tau,\theta)$, we deduce that~$r_\ell(\mathcal{W})=W(\widetilde{\pi},r_\ell(\theta^{-1}))$. 
\end{proof}

We also introduce ``standard lifts'' of representations of Whittaker type. We will see in Theorem 
\ref{integral-whittaker} that they are Whittaker lifts.

\begin{df}\label{standard-lift}
If~$\pi= \L(\Delta_1)\times \cdots \times \L(\Delta_t)$ is an~$\ell$-modular representation of Whittaker type, we will call a \emph{standard lift} of~$\pi$ a representation~$\tau=\L(D_1)\times \cdots \times \L(D_t)$, where~$\L(D_i)$ lifts~$\L(\D_i)$ as in Remark \ref{lift-segments}.
\end{df}

\begin{rem}
Again, standard lifts are not strictly speaking lifts.  However, as parabolic induction commutes with reduction modulo~$\ell$, if~$\tau$ is a standard lift of an~$\ell$-modular representation~$\pi$, then~$r_\ell(\tau)$ is the semi-simplification of~$\pi$.  
\end{rem}

%
We now show that standard lifts are Whittaker lifts.
 
\begin{thm}\label{integral-whittaker}
Let~$\pi= \L(\Delta_1)\times \cdots \times \L(\Delta_t)$ be an~$\ell$-modular representation of Whittaker type, and~$\tau=  \L(D_1)\times \cdots \times \L(D_t)$ be a standard lift of~$\pi$, then it is a Whittaker lift. 
\end{thm}
\begin{proof}
If~$\tau$ is generic, it is shown in \cite[Theorem 2]{V2} that~$W_e(\tau,\theta)$ is an integral structure in~$W(\tau,\theta)$, such that~$W(\pi,r_{\ell}(\theta))=W_e(\tau,\theta)\otimes_{\Zl}\Fl$. We use this result together with the properties of parabolic induction with respect to lattices, and a result from \cite{CS} about the explicit description of Whittaker functionals on induced representations.

For each~$1\leqslant i\leqslant t$, each~$\L(D_i)$ is generic and hence~$W_e(\L(D_i),\theta)$ is an integral structure in~$W(\L(D_i),\theta)$. By \cite[I  9.3]{V}, the lattice~$\mathfrak{L}=W_e(\L(D_1),\theta)\times \dots \times W_e(\L(D_t),\theta)$ is an integral structure in~$\tau$. The space~$\mathfrak{L}$  consists of all smooth functions from~$G$ to~$L=W_e(\L(D_1),\theta)\otimes \dots \otimes W_e(\L(D_t),\theta)$, with the tensor product over~$\Zl$, which transform on the left by~$\L(D_1)\otimes\cdots \otimes \L(D_t)$. We recall that~$W(\L(D),\theta)=W(\L(D_1),\theta)\otimes \dots \otimes W(\L(D_t),\theta)$, as well as~$W(\L(\Delta),\theta)=W(\L(\D_1),\theta)\otimes \dots \otimes W(\L(\D_t),\theta)$, is a representation of a standard parabolic subgroup~$Q=MU$ of~$G_n$, trivial on~$U$.

A function~$f$ in~$\tau=W(\L(D_1),\theta)\times \dots \times W(\L(D_t),\theta)$, by definition of parabolic induction, is a 
map from~$G_n$ to~$W(\L(D),\theta)$, i.e. for~$g\in G_n$,~$f(g)\in W(\L(D),\theta)$ identifies with a map from~$M$ to~$\Ql$, so we can view~$f$ as a map of two variables from~$G_n\times M$ to~$\Ql$, and similarly, we can view the elements in~$\pi=W(\L(\D_1),\theta)\times \dots \times W(\L(\D_t),\theta)$ as maps from~$G_n\times M$ to~$\Fl$. In \cite[Corollary 2.3]{CS}, it is shown (for minimal parabolics, but their method works for general parabolics), that there is a Weyl element~$w$ in~$G_n$, such that if one takes~$f\in \tau$, then there is a compact open subgroup~$U_f$ of~$U$ which satisfies that for any compact subgroup~$U'$ of~$U$ containing~$U_f$, the integral~$\int_{U'} f(w u,1_M)\theta^{-1}(u)du$ is independent from~$U'$. We 
will write~$\l(f)=\int_{U} f(w u, 1_M)\theta^{-1}(u)du$. This assertion is also true for~$\pi$ with the same proof, for the same choice of~$w$, we write~$\mu(h)=\int_{U} h(w u, 1_M)r_\ell(\theta)^{-1}(u)du$ for~$h\in \pi$. Both~$\l$ and~$\mu$ are nonzero Whittaker functionals on~$\tau$ and~$\pi$, respectively, and~$\l$ sends~$\mathfrak{L}$ to~$\Zl$  for a correct normalisation of~$du$. We can moreover suppose, for correct normalisations of the~$\ell$-adic and the~$\ell$-modular Haar measures~$du$, that~$\mu=r_\ell(\l)$. The surjective map~$w: \tau \mapsto  W(\tau,\theta)$ which takes~$f$ to~$W_f$, defined by~$W_f(g)=\l(\tau(g)f)$, sends~$\mathfrak{L}$ into~$W_e(\tau,\theta)$.  Similarly, for~$h\in\pi$, if we set~$W_h(g)=\mu(\pi(g)h)$, then the map~$ \pi \mapsto W(\pi,r_\ell(\theta))$, taking~$h$ to~$W_h$, is surjective, and we have~$r_\ell(W_f)=W_{r_\ell(f)}$.  From this, we obtain that~$\mathcal{W}=w(\mathfrak{L})$ is a sublattice of~$W_e(\tau,\theta)$ (see \cite{V}, I 9.3), and~$r_\ell(\mathcal{W})=W(\pi,\theta)$. 
\end{proof}

\subsection{Generic representations}\label{genericsection}

It is stated (in different terms) in \cite[1.8.4]{Viginv} that if~$\tau$ is an integral generic~$\ell$-adic representation of~$G_n$, then~$r_\ell(\tau)$ has a unique generic component~$J_{\ell}(\tau)$, and that the map~$\tau\mapsto J_\ell(\tau)$ is a surjection from the set of classes of~$\ell$-adic generic representations of~$G_n$ to the set of classes of~$\ell$-modular generic representations of~$G_n$.  We recall why this is true (it also follows at once from the fact that the derivatives commute with reduction modulo~$\ell$, but we shall go in the other direction and use this result to prove that derivatives commute with reduction modulo~$\ell$). 

\begin{LM}\label{Jl}
Let~$\tau$ be an~$\ell$-adic integral generic representation of~$G_n$, then~$r_\ell(\tau)$ contains (with multiplicity one) a unique generic component~$J_{\ell}(\tau)$. Moreover, if~$r_1\times \dots \times r_t$ is a product of integral cuspidal~$\ell$-adic representations such that~$r_1\times \dots \times r_t \twoheadrightarrow \tau$, then setting~$\rho_i=r_\ell(r_i)$, the representation~$J_{\ell}(\tau)$ is the unique generic subquotient of~$\rho_1\times \dots \times \rho_t$.
\end{LM}
\begin{proof}
To prove that~$r_\ell(\tau)$ contains a unique generic component, as in the start of the proof of Theorem~\ref{integral-whittaker} we call upon \cite[Theorem 2]{V2}, which tells us that~$W_e(\tau,\theta)$ is an integral structure in~$\tau$. The representation~$W_e(\tau,\theta)\otimes_{\Zl}\Fl$ is a non-zero~$G_n$-submodule of~$\mathcal{C}^\infty(N_n\backslash G_n,r_\ell(\theta))$, and this implies that~$r_\ell(\tau)$ contains at least one generic subquotient~$J_{\ell}(\pi)$.  Let~$p: \sigma=r_1\times \dots \times r_t \twoheadrightarrow \tau$ be a surjection as in the statement (provided by considering the cuspidal support of~$\tau$), then by multiplicity one of Whittaker functionals,~$p$ factors through a surjection~$s:W(\sigma,\theta) \twoheadrightarrow \tau$ as~$\tau$ is generic. 
By Theorem \ref{integral-whittaker}, there is a lattice~$\mathcal{W}_\sigma\subset W_e(\sigma,\theta)$ such that if we set~$\rho_i=r_\ell(r_i)$, the~$\Fl$-module~$\l=\mathcal{W}_\sigma\otimes \Fl$ is equal to~$W(\rho_1\times \dots \times \rho_t,r_\ell(\theta))$. By \cite[9.3 (vi)]{V}, the~$\Zl$-module~$\mathfrak{L}=s(\mathcal{W}_\sigma) \subset \tau$ is a lattice in~$\tau$, and~$\mathfrak{L}\otimes \Fl$ is a quotient of~$\l$. By hypothesis~$J_{\ell}(\pi)$ is an 
irreducible subquotient of~$\mathfrak{L}\otimes \Fl$, hence of~$\l$, i.e. of~$W(\rho_1\times \dots \times \rho_t,r_\ell(\theta))$, and hence of~$\rho_1\times \dots \times \rho_t$. In particular~$J_{\ell}(\pi)$ is the unique generic subquotient of~$\rho_1\times \dots \times \rho_t$, and this implies that~$r_\ell(\tau)$ contains a unique generic subquotient.
\end{proof}

We now prove that derivatives commute with reduction modulo~$\ell$. First we recall some basic facts. Let~$(\tau,V)$ be an integral~$\ell$-adic representation of finite length, with filtration~$0\subset V_1\subset \dots \subset V_n=V$. Then it is a consequence of \cite[I 9.3]{V}, that the subquotients~$V_{i+1}/V_i$ are integral, and~$r_\ell(V)=\bigoplus_i r_\ell(V_{i+1}/V_i)$. 

\begin{prop}\label{reduction-commutes-with-derivatives-1} Let~$\tau$ be an~$\ell$-adic integral representation of~$G_n$ of finite length, then we have~$(\Phi^-)^n r_\ell(\tau)=r_\ell((\Phi^-)^n \tau)$. 
\end{prop}
\begin{proof}
As~$\Phi^-$ is an exact functor, we can assume that~$\tau$ is irreducible. In this case, our assertion amounts, by multiplicity one of Whittaker functionals, to show that~$(\Phi^-)^n r_\ell(\tau)\simeq \Fl$ if~$\tau$ is generic, and~$(\Phi^-)^n r_\ell(\tau)= \{0\}$ otherwise. If~$\tau$ is generic, this follows from Lemma \ref{Jl}. If not, then, considering cuspidal supports again, one has~$\tau\hookrightarrow \mu_1\times \dots \times \mu_r$, where the~$\mu_i$ are cuspidal. We also have a surjection~$s:\mu_1\times \dots \times \mu_r \twoheadrightarrow W(\mu_1\times \dots \times \mu_r,\theta)$, and denoting by~$\tau'$ its kernel, we have~$\tau$ is contained in $ \tau'$. In particular,~$r_{\ell}(\tau)$ is contained in $r_{\ell}(\tau')$. Moreover, we have~$r_{\ell} (\mu_1\times \dots \times \mu_r)=r_\ell (\tau')\oplus r_\ell(W(\mu_1\times \dots \times \mu_r,\theta))$. On the other hand, setting~$\rho_i=r_\ell(\mu_i)$, then~$r_\ell(\mu_1\times \dots \times \mu_r)$ is the semi-simplification of~$\rho_1\times \dots \times \rho_r$, in particular, it contains a unique generic component. But, according to Theorem \ref{integral-whittaker},~$r_\ell(W(\mu_1\times \dots \times \mu_r,\theta))$ contains the semi-simplification of~$W(\rho_1\times \dots \times \rho_r,\theta)$, and in particular it contains a generic component, namely the unique generic component of~$r_\ell(\mu_1\times \dots \times \mu_r)$, so that~$r_\ell(\tau')$, hence~$r_\ell(\tau)$, does not. Finally, 
we have~$(\Phi^-)^n (r_\ell(\tau))=\{0\}$.
\end{proof}

As an immediate corollary, we obtain the following.

\begin{cor}\label{reduction-commutes-with-derivatives-2}
Let~$\pi$ be an integral~$\ell$-adic representation of~$G_{n_1}\times G_{n_2}$ of finite length. Then 
$(\Id\otimes (\Phi^-)^{n_2}) r_\ell(\pi)= r_\ell((\Id\otimes (\Phi^-)^{n_2})\pi)$.
\end{cor}

From this corollary follows the proof of Theorem \ref{derivatives-commute}.  Indeed, if we set~$G=G_n$ and~$P=P_{(n-k,k)}$, and~$\tau$ as in the statement of  Theorem \ref{derivatives-commute}, we have~$\tau^{(k)}=(\Id\otimes (\Phi^-)^{k})\circ \boldsymbol{r}_P^G(\tau)$ (cf. for example the proof of \cite[Proposition 2.5]{Matringe}). But~$\boldsymbol{r}_P^G$ commutes with reduction modulo~$\ell$, as we recalled in Section \ref{sectintegral}, and so does~$\Id\otimes (\Phi^-)^{k}$, according to Corollary \ref{reduction-commutes-with-derivatives-2}, and this completes the proof of Theorem \ref{derivatives-commute}.

We now recall \cite[Theorem V.7]{V3} (cf. also \cite[Theorem 9.10]{MS}, and \cite[Theorem 9.7]{Z} for the original proof when~$R=\mathbb{C}$), which classifies generic representations.

\begin{thm}\label{gen} Let~$\pi=\L(\D_1)\times \dots \times \L(\D_t)$ be an~$\ell$-modular (or~$\ell$-adic) representation of Whittaker type of~$G_n$, then~$\pi$ is irreducible if and only if the segments~$\D_i$ and~$\D_j$ are unlinked, for all~$i,j\in\{1,\ldots, t\}$ with~$i\neq j$. In this case, the product is commutative and up to order, and the segments which appear, together with their multiplicities, are determined by~$\pi$. 
\end{thm}

First, we notice the following consequence.

\begin{rem}\label{generic-lifts}
If~$\pi$ is an~$\ell$-modular generic representation of~$G_n$, as parabolic induction commutes with reduction modulo~$\ell$, it is easy to see that it admits standard lifts which are generic (i.e. irreducible). In this case, one has~$r_\ell(\tau)=\pi$ (because~$\pi$ is irreducible), and~$\pi$ is a true lift, which is generic. Any generic lift~$\tau$ is such that~$r_\ell(\tau)=J_\ell(\tau)=\pi$. In fact, thanks Theorem \ref{derivatives-commute}, any lift of~$\pi$ (which is necessarily irreducible), is generic.
\end{rem}

As a consequence of Theorem \ref{gen}, we observe that a generic representation~$\pi$ of~$G_n$ can be written in a unique way as the product of two generic representations 
\[\pi_b\times \pi_{tnb},\] where~$\pi_b$ is a product of banal segments, and~$\pi_{tnb}$ a product of non-banal segments. We first notice that~$\pi_b$ is a banal representation in the sense of \cite[Remarque 8.15]{MS}.  In fact, it is the largest banal subproduct of~$\pi$, this claim following at once from the next proposition. Similarly~$\pi_{tnb}$ is non-banal, and has no nontrivial banal subproduct. For this reason we will call~$\pi_b$ the \emph{banal part} of~$\pi$, and we will call~$\pi_{tnb}$ the \emph{totally non-banal part} of~$\pi$.

\begin{prop}\label{banal-part}
Let~$\rho$ be a cuspidal~$\ell$-modular representation of~$G_r$, and~$\pi$ be a generic representation of~$G_n$ 
whose cuspidal support is contained in~$\Z_\rho$. Then~$\pi$ is banal if and only if~$\rho$ is banal.
\end{prop}
\begin{proof}
If~$e(\rho)=\ell$, then both~$\pi$ and~$\rho$ are non-banal, so we suppose that~$e(\rho)\neq \ell$, in which case~$e(\rho)=|\Z_\rho|\geqslant 2$. It is convenient to think of~$\Z_\rho$ as a circle, with~$e(\rho)$ vertices drawn on it. If~$\rho$ is non-banal, then it is obvious that~$\pi$ is non-banal. Conversely suppose that~$\pi$ is non-banal, but~$\rho$ is banal. As~$\pi$ is generic, it is a product of unlinked generic segments.  By hypothesis, the cuspidal support of~$\pi$ is a subset of~$\Z_\rho$, this means that if one takes two segments occurring in~$\pi$, they are either included in one another, or disjoint and not juxtaposed. As~$\pi$ is non-banal, all the vertices on the circle~$\Z_\rho$ must be covered by the union of all segments occurring in~$\pi$. This implies that either one segment is of length~$\geqslant e(\rho)$, or that two segments are linked. Both cases are impossible since~$\pi$ is generic, a contradiction.
\end{proof}

\subsection{Lifting supercuspidal representations}\label{liftingsection}
The aim of this subsection is to show that a non-banal supercuspidal~$\ell$-modular representation has two lifts which are not isomorphic after twisting by an unramified character. A fact that we will use to compute~$L$-factors of non-banal supercuspidal~$\ell$-modular representations.  To prove this, we use the Bushnell--Kutzko construction to reduce the problem to finite group theory.

We first recall the Green classification of all cuspidal~$\ell$-adic representations of~$\GL_m(k_F)$, where~$k_F$ is a finite field of characteristic~$p$, and the associated results of James on reduction modulo~$\ell$. Let~$k_E$ be a finite extension of~$k_F$ of degree~$m$.  The group~$\Hom(k_E^\times ,\Ql^\times )$ is cyclic of order~$|k_E^\times |$, and is the direct product of its~$\ell$-singular part~$\Hom(k_E^\times ,\Ql^\times )_s$ and its~$\ell$-regular part~$\Hom(k_E^\times ,\Ql^\times )_r$. The map~$r_{\ell
}$ is a surjective morphism 
\[\Hom(k_E^\times ,\Ql^\times )\rightarrow \Hom(k_E^\times ,\Fl^\times )\]
with kernel~$\Hom(k_E^\times ,\Ql^\times )_s$. Let~$R=\Ql$ or~$\Fl$, a character~$\chi\in \Hom(k_E^\times ,R^\times )$ is called \emph{$(k_E/k_F)$-regular} if its~$\Gal(k_E/k_F)$-orbit is of cardinality~$|\Gal(k_E/k_F)|$, i.e. if~$\chi$ does not factor through the norm of any proper intermediate extension. If we write~$G_\chi$ for the fixator of~$\chi$ in~$\Gal(k_E/k_F)$, it is equivalent to say that~$G_\chi=\{1\}$.

\begin{thm}[{\cite{Green},~\cite{James}}]\label{Greenjames}
Let~$k_F$ be a finite field, and~$k_E$ be an extension of~$k_F$ of degree~$m$.
\begin{enumerate}
\item There is a surjective map 
\begin{align*}
\mu&\mapsto \sigma(\mu)
\end{align*}
from the set of~$(k_E/k_F)$-regular characters of~$k_E^\times~$ to the set isomorphism classes of cuspidal~$\ell$-adic representations of~$\GL_m(k_F)$, such that the preimage of 
$\sigma(\mu)$ is the~$\Gal(k_E/k_F)$-orbit of~$\mu$. 
\item The reduction~$r_\ell(\sigma(\mu))$ is cuspidal, and~$r_\ell(\sigma(\mu))=r_\ell(\sigma(\mu'))$ if and only if 
the~$\ell$-regular parts of~$\mu$ and~$\mu'$, as elements of the group~$\Hom(k_E^\times ,\Ql^\times )$, are conjugate under~$\Gal(k_E/k_F)$. Moreover 
$r_\ell(\sigma(\mu))$ is supercuspidal if and only if the~$\ell$-regular part of~$\mu$ is~$(k_E/k_F)$-regular.
Finally, any cuspidal~$\ell$-modular representation of~$\GL_m(k_F)$ lifts to a cuspidal~$\ell$-adic representation~$\sigma(\mu)$. \end{enumerate}
\end{thm}

Let~$\chi\in \Hom(k^\times,\Ql)$,~$\chi_r$ be the~$\ell$-regular part of~$\chi$, and~$\chi_s$ be the singular part which has order a power of~$\ell$, so that we can write~$\chi=\chi_s\chi_r$ uniquely.  As~$G_\chi=G_{\chi_s}\cap G_{\chi_r}$, we notice the following obvious fact:

\begin{LM}
The~$\chi$ character is~$(k_E/k_F)$-regular if and only if~$G_{\chi_s}\cap G_{\chi_r}=\{1\}$.
\end{LM}

We now make a first step towards computing the~$L$-factors of non-banal supercuspidal representations.  We continue with the notations of Theorem~\ref{Greenjames}

\begin{prop}\label{lifts}
Let~$\overline{\sigma}$ be a cuspidal~$\ell$-modular representation of~$\GL_m(k_F)$, and~$\sigma(\chi)$ be a cuspidal lift of~$\overline{\sigma}$. Then if the generators of the cyclic~$\ell$-group~$\Hom(k_E^\times ,\Ql^\times )_s$ are not in the same~$G_{\chi_r}$-orbit, then~$\overline{\sigma}$ has two non-isomorphic cuspidal lifts. This implies that if~$\ell$ divides~$|k_E^\times|$ and~$\overline{\sigma}$ is supercuspidal, then it has two non-isomorphic lifts.
\end{prop}
\begin{proof}
Let~$\sigma(\chi)$ be a cuspidal representation lifting~$\overline{\sigma}$ and write~$\chi=\chi_s\chi_r$. As~$\chi$ is regular, one has~$G_{\chi_s}\cap G_{\chi_r}=\{1\}$. But if~$\mu$ is any generator of~$\Hom(k_E^\times ,\Ql^\times )_s$, one has~$G_\mu\subset G_{\chi_s}$, so that~$\mu\chi_r$ is regular as well. 
Moreover~$\sigma(\mu\chi_r)$ lifts~$\overline{\sigma}$. But if two generators~$\mu$ and~$\mu'$ of~$\Hom(k_E^\times ,\Ql^\times )_s$ are in different~$G_{\chi_r}$-orbits, then clearly~$\chi_1=\mu\chi_r$ and~$\chi_2=\mu'\chi_r$ are in different~$\Gal(k_E/k_F)$-orbits, and~$\sigma(\chi_1)$ and ~$\sigma(\chi_2)$ are two non-isomorphic lifts of~$\overline{\sigma}$. Now if~$\overline{\sigma}$ is supercuspidal, then the character~$\chi_r$ is~$(k_E/k_F)$-regular, hence~$G_{\chi_r}$ is trivial. Hence our second assertion follows when~$|\Hom(k_E^\times ,\Ql^\times )_s|\geqslant 3$, because in this case~$\Hom(k_E^\times ,\Ql^\times )_s$ has at least~$2$ generators.  Now when~$|\Hom(k_E^\times ,\Ql^\times )_s|=2$, the generator~$\mu$ of~$\Hom(k_E^\times ,\Ql^\times )_s$ is the only element of order~$2$ in~$\Hom(k_E^\times ,\Ql^\times )$, hence it is fixed by~$\Gal(k_E/k_F)$. In particular~$\mu\chi_r$ is regular as well, but in a different~$\Gal(k_E/k_F)$-orbit, and we conclude as before. 
\end{proof}

\begin{rem}
Notice that in contrast to Proposition~\ref{lifts} a non-banal cuspidal non-supercuspidal representation may only have one lift, for example for any~$\ell\neq 2$,~$\GL_2(\mathbb{F}_2)$ has only one isomorphism class of cuspidal~$\ell$-adic representation~$\sigma$, which corresponds 
to unique~$\Gal(\mathbb{F}_4/\mathbb{F}_2)$-orbit any non-trivial characters of~$\mathbb{F}_4^\times \simeq \Z/3\Z$. Hence any cuspidal~$\ell$-modular
representation of~$\GL_2(\mathbb{F}_2)$ has exactly one lift. The only~$\ell$ which divides~$\mathbb{F}_4^\times~$ is~$3$, in which case 
$r_3(\sigma)$ is cuspidal non-supercuspidal.
\end{rem}

In~\cite[Chapitre~III]{V}, the Bushnell--Kutzko construction of all cuspidal~$\ell$-adic representations of~$G_n$, is extended to~$\ell$-modular representations. Let~$R=\Ql$ or~$\Fl$, and~$\pi$ be a cuspidal~$R$-representation of~$G_n$.
\begin{enumerate}
\item By \cite[Chapter 6]{BK93},~\cite[Chapitre III]{V}, the representation~$\pi$ contains a \emph{maximal extended simple type}~$(\bJ,\Lambda)$,  ($\bJ$ is a compact mod-centre open subgroup of~$G_n$ and~$\Lambda$ is an irreducible~$R$-representation of~$\bJ$ of finite dimension) and we have~$\pi\simeq\ind_{\bJ}^{G_n}(\Lambda)$.  Moreover, any representation compactly induced from an extended maximal simple type is cuspidal and a cuspidal representation determines a maximal simple type~$(\bJ,\Lambda)$ up to conjugation in~$G_n$ (the so-called intertwining implies conjugacy property). 
\item Let~$(\bJ,\Lambda)$ be an extended maximal simple type for~$\pi$. The group~$\bJ$ is defined via a field extension~$E=F[\beta]$ of~$F$ in~$\mathcal{M}_n$ such that~$\bJ=E^\times\ltimes J$ with~$J$  the maximal compact subgroup of~$\bJ$. The restriction~$\l$ of~$\Lambda$ to~$J$ decomposes (non-uniquely) as~$\l= \kappa\otimes \sigma$, with~$\kappa$ a~\emph{$\beta$-extension}, and~$\sigma$ an irreducible cuspidal~$R$-representation of~$J/J^1\simeq G_m(k_E)$.  The pair~$(J,\lambda)$ is called a \emph{maximal simple type}.  By~\cite[Lemmas 6.1~and~6.8]{MS}, the representation~$\pi$ is supercuspidal if and only if~$\sigma$ is supercuspidal.  
\item The \emph{ramification index}~$e$ of~$\pi$, is the ramification index of the extension~$E/F$ above, which is well defined as intertwining of extended maximal simple types implies conjugacy. 
\item Let~$\pi$ be a cuspidal~$\ell$-modular representation and~$e$ be the ramification index of~$\pi$.  By~\cite[Lemme~5.3~and~Section~6]{MS}, the representation~$\pi$ is banal if and only if~$q^{n/e}\not\equiv 1[\ell]$ and, a banal cuspidal~$\ell$-modular representation is necessarily supercuspidal.   
\item Let~$\pi$ be a cuspidal~$\ell$-modular representation.  If~$\pi$ contains an extended maximal simple type~$(\bJ,\Lambda)$, then one can lift~$(\bJ,\Lambda)$ to an integral~$\ell$-adic maximal extended type~$(\bJ,\widetilde{\Lambda})$, so that the representation~$\tau=\ind_{\bold{J}}^{G_n}\widetilde{\Lambda}$ is a cuspidal lift of~$\pi$ and all lifts are of this form. Moreover, if~$\lambda=\Lambda\mid_J=\kappa\otimes\sigma$, then~$\kappa$ lifts to a~$\beta$-extension~$\widetilde{\kappa}$ and~$\sigma$ lifts to a cuspidal representation~$\widetilde{\sigma}$, so that~$(J,\widetilde{\kappa}\otimes\widetilde{\sigma})$ is a maximal simple type in~$(\bJ,\widetilde{\Lambda})$
\end{enumerate}
\begin{rem}
For a cuspidal~$R$-representation~$\pi$ of~$G_n$, the integer~$f(\pi)$ we mentioned in Section \ref{repwhit}, is defined by~$f(\pi)=n/e$ where~$e$ is the ramification index of~$\pi$.
\end{rem}

We can now prove the result on lifting supercuspidal~$\ell$-modular representations of~$G_n$ we need later:
\begin{prop}\label{supercuspliftingnoniso}
Let~$\pi$ be a non-banal supercuspidal~$\ell$-modular representation of~$G_n$.  There exists two lifts of~$\pi$ which are not isomorphic by twisting by unramified characters.
\end{prop}

\begin{proof}
Let~$(J,\kappa\otimes \sigma)$ be a simple type contained~$\pi$, then~$\sigma$ is a supercuspidal representation of~$J/J^1\simeq GL(m,k_E)$. By Proposition~\ref{lifts}, it has two non-isomorphic cuspidal lifts~$\tau$ and~$\tau'$. Let~$\widetilde{\kappa}$ be a~$\beta$-extension lifting~$\kappa$ such that~$(J,\widetilde{\kappa}\otimes \tau)$ and~$(J,\widetilde{\kappa}\otimes \tau')$ are simple types for two cuspidal lifts of~$\pi$ which we write~$\rho$ and~$\rho'$, then the representations~$\rho$ and~$\rho'$ can not be in the same orbit under twisting by unramified characters.
\end{proof}

\section{Rankin--Selberg local factors for representations of Whittaker type}\label{SectionRSdefs}

The theory of derivatives (\cite{BZ0} and \cite{BZ}) being valid in positive characteristic (see Subsection \ref{der}) and equipped with the theory of~$R$-Haar measures (see Subsection \ref{haarmeasures}), means we can now safely follow \cite{JPS2} to define~$L$-factors and~$\gamma$-factors. However, as we do not have a Langlands' quotient theorem at our disposal, which would allow us to associate to an irreducible representation of~$G_n$, a unique representation with an injective Whittaker model lying above it, we restrict our attention to representations of Whittaker type (see Subsection \ref{repwhit}).  Of independent interest, we give a different shorter proof than [ibid.] on multiplicity one of~$\gamma$-factors in Proposition \ref{mult1}, relying on derivatives rather than invariant distributions.  In the main result of the section, Theorem \ref{compat1}, we show that the~$\ell$-modular~$\gamma$-factor of a pair of representations always equals the reduction modulo~$\ell$ of the~$\ell$-adic~$\gamma$-factor of any Whittaker lifts, and that the~$L$-factor of this pair divides the reduction modulo~$\ell$ of the~$\ell$-adic~$L$-factor of any Whittaker lifts. 
 
\subsection{Definition of the L-factors}

We first recall the asymptotics of Whittaker functions obtained in \cite[Proposition 2.2]{JPS1}. For~$1\leqslant i\leqslant n$, we write~$Z_i$ for subgroup~$\{\diag(tI_i,I_{n-i}),t\in F^\times\}$ of~$G_n$. The diagonal torus~$A_n$ of~$G_n$ is the direct product~$Z_1\times\cdots\times Z_n$.

\begin{LM}\label{asymptotics}
 Let~$\pi$ be an~$R$-representation of Whittaker type of~$G_n$. For each~$1\leqslant i\leqslant n-1$, 
there is a finite family~$X_i(\pi)$ of characters of~$Z_i$, such that if~$W$ is a Whittaker function in~$W(\pi,\theta)$, then its restriction~$W(z_1,\dots,z_{n-1})$ to~$A_{n-1}=Z_1\times\cdots\times Z_{n-1}$ is a sum of functions of the form \[\Phi(z)\prod_{i=1}^{n-1}\chi_i(z_i)v(z_i)^{m_i}\] for 
$\chi_i\in X_i(\pi)$, integers~$m_i\geqslant 0$, and~$\Phi\in \mathcal{C}_c^\infty(F^{n-1})$. 
\end{LM}
The proof of Jacquet--Piatetski-Shapiro--Shalika in [ibid.] applies \emph{mutatis mutandis} for~$\ell$-modular representations.

\begin{rem}
 For~$1\leqslant i \leqslant n-1$, we can take~$X_i(\pi)$ to be the family of central characters (restricted to~$Z_i$) of the irreducible components of the (non-normalised) Jacquet module~$\pi_{N_{i,n-i}}$. We denote by~$X_n(\pi)$ the singleton~$\{\omega_\pi\}$. 
We denote by~$E_i(\pi)$, the family of central characters (restricted to~$Z_i$) of the irreducible components of the normalised Jacquet module~$\pi_{N_{i,n-i}}$, for~$1\leqslant i\leqslant n-1$, and let~$E_n(\pi)=X_n(\pi)$.  The family~$E_i(\pi)$ is obtained from~$X_i(\pi)$ by multiplication by an unramified character of~$Z_i$, in particular, if~$R=\Ql$, the characters in~$E_i(\pi)$ are integral if and only if those in~$X_i(\pi)$ are integral.
\end{rem}

\begin{prop}\label{integralsmakesense}
Let~$\pi$ be an~$R$-representation of Whittaker type of~$G_n$, and~$\pi'$ an~$R$-representation of Whittaker type of~$G_m$, for~$m\leqslant n$. 
\begin{itemize}
 \item The case~$n=m$. Let~$W\in W(\pi,\theta)$,~$W'\in W(\pi',\theta^{-1})$, and~$\Phi \in \mathcal{C}_c^\infty(F^n)$. 
Then, for every~$k\in \Z$, the coefficient \[c_k(W,W',\Phi)=\int_{N_n\backslash G_n^k} W(g)W'(g)\Phi(\eta g)dg\] is well-defined, and vanishes for~$k<<0$. 
\item The case~$m\leqslant n-1$. For~$0\leqslant j \leqslant n-m-1$, let~$W\in W(\pi,\theta)$, and 
$W'\in W(\pi',\theta^{-1})$. Then for every~$k\in \Z$, the coefficient 
\[c_k(W,W';j)=\int_{\mathcal{M}_{j,m}}\int_{N_m\backslash G_m^k} W\begin{pmatrix} g &  &  \\ x & I_j & \\
 & & I_{n-m-j}\end{pmatrix} W'(g)dgdx\] is well-defined, 
and vanishes for~$k<<0$.  When~$m=n-1$, we will simply write~$c_k(W,W')$ for~$c_k(W,W';0)$.
\end{itemize}
\end{prop}
\begin{proof}
The only thing to check is that the coefficients in the statement are well defined, i.e. finite sums, and zero for~$k$ negative enough. 
This is a consequence of Iwasawa decomposition together with Corollary \ref{iwasplit}, and Lemma \ref{asymptotics} for the case~$m\geqslant n-1$, 
and, in the case~$m\leqslant n-2$, that the map~$W\begin{pmatrix} g &  &  \\ x & I_j & \\
 & & I_{n-m-j}\end{pmatrix}$ has compact support with respect to~$x$, independently of~$g$, by  \cite[Lemma 6.2]{JPS2}.  
\end{proof}

Following Proposition \ref{integralsmakesense}, we now can define our \emph{Rankin--Selberg formal series}. 

\begin{df}
\begin{itemize}
 \item The case~$n=m$. Under the same notation as Proposition \ref{integralsmakesense}, we define the following formal Laurent series
\[I(X,W,W',\Phi)=\sum_{k\in \Z} c_k(W,W',\Phi)X^k \in R((X)).\]
\item The case~$m\leqslant n-1$.  Under the same notation as Proposition \ref{integralsmakesense}, we define the following formal Laurent series
\[I(X,W,W';j)=\sum_{k\in \Z} c_k(W,W';j)q^{k(n-m)/2}X^k \in R((X)).\]
When~$m=n-1$, we will simply write~$I(X,W,W')$ for~$I(X,W,W';0)$.
\end{itemize}
\end{df}

The~$L$-factors we study are defined by the following theorem.

\begin{thm}\label{Lfactorsdefined}
Let~$\pi$ be an~$R$-representation of Whittaker type of~$G_n$, and~$\pi'$ an~$R$-representation of Whittaker type of~$G_m$, for~$1\leqslant m\leqslant n$.
\begin{itemize}
 \item If~$n=m$, the~$R$-submodule spanned by the Laurent series~$I(X,W,W',\Phi)$ as~$W$ varies in~$W(\pi,\theta)$, 
~$W'$ varies in~$W(\pi',\theta^{-1})$, and~$\Phi$ varies in~$\mathcal{C}_c^\infty(F^{n})$, is 
a fractional ideal~$I(\pi,_pi')$ of~$R[X^{\pm 1}]$, and it has a unique generator which is an Euler factor~$L(X,\pi,\pi')$.
\item If~$1\leqslant m\leqslant n-1$, fix~$j$ between~$0$ and~$n-m-1$. The~$R$-submodule spanned by the Laurent series~$I(X,W,W';j)$ as~$W$ varies in~$W(\pi,\theta)$,~$W'$ varies in~$W(\pi',\theta^{-1})$, is 
a fractional ideal~$I(\pi,\pi')$ of~$R[X^{\pm 1}]$, is independent of~$j$, and it has a unique generator which is an Euler factor~$L(X,\pi,\pi')$.
\end{itemize}
\end{thm}
\begin{proof}
We treat the case~$m\leqslant n-2$, the case~$m\geqslant n-1$ is totally similar. First we want to prove that our formal series in fact belong 
to~$R(X)$. In this case, the coefficient~$c_k(W,W';j)$ is equal to 
 \[\int_{(K_m\cap B_m)\backslash K_m}\int_{A_m^k}\int_{\mathcal{M}_{j,m}} W\begin{pmatrix} ask &  &  \\ x & I_j & \\
 & & I_{n-m-j}\end{pmatrix} W'(ask)\d_{B_m}(a)^{-1}dxdadk,\] which, by smoothness of~$W$ and~$W'$, we can write as a finite sum
 \[\sum_i \int_{A_m^k} W_i\begin{pmatrix} a & \\ & I_{n-m}\end{pmatrix} W'_i(a)\d_{B_m}(a)^{-1}da,\]
 with functions~$W_i\in W(\pi,\theta)$ and~$W'_i\in W(\pi',\theta^{-1})$. For~$W$ and~$W'$ in~$W(\pi,\theta)$, let
 \[b_k(W,W')=\int_{A_m^k} W\begin{pmatrix} a & \\ & I_{n-m}\end{pmatrix} W'(a)\d_{B_m}(a)^{-1}da,\] it is thus enough to check that 
 \[J(W,W')=\sum_{k\in \Z} b_k(W,W')q^{(n-m)/2}X^k\] belongs to~$R(X)$. Following the proof of \cite{JPS2} we see that, by Lemma \ref{asymptotics}, they belong to~$\frac{1}{P(X)}R[X^{\pm 1}]$, where~$P(X)$ is a suitable power of the product over the unramified characters
~$\chi_i$  in~$E_i(\pi)$ for~$1\leqslant i \leqslant n$ and the unramified characters~$\mu_j$  in~$E_j(\pi')$ for~$1\leqslant j \leqslant m$ of the Tate~$L$-factors~$L(X,\chi_i\mu_j)$. By \cite{M}, this factor is equal to~$1$ if~$R=\Fl$, and~$q \equiv 1[\ell]$, and is equal to~$1/(1-\chi_i\mu_j(\w)X)$ otherwise.  The other properties follow immediately from \cite{JPS2}. 
\end{proof}

The proof of Theorem \ref{Lfactorsdefined} implies the following corollary.

\begin{cor}\label{div} 
\begin{itemize}
\item If~$\pi$ and~$\pi'$ are~$\ell$-adic representations of Whittaker type of~$G_n$ and~$G_m$, then~$1/L(X,\pi,\pi')\in\Zl[X]$.
\item If~$\pi$ and~$\pi'$ are~$\ell$-modular representations of Whittaker type of~$G_n$ and~$G_m$, and~$q\equiv 1 [\ell]$, then~$L(X,\pi,\pi')=1$.
\end{itemize}
\end{cor}
\begin{proof} By our assertion at the end of the proof of Theorem \ref{Lfactorsdefined}, the polynomial~$\mathcal{Q}=1/L(X,\pi,\pi')$ divides (in~$R[X^{\pm 1}]$, hence in 
$R[X]$) a power of the product~$\mathcal{P}$ of the polynomials~$1/L(X,\chi_i\mu_j)$ over the set of unramified characters~$\chi_i$ in~$E_i(\pi)$ for~$1\leqslant i \leqslant n$ and unramified characters~$\mu_j$  in~$E_j(\pi')$ for~$1\leqslant j \leqslant m$.  We already noticed that~$\mathcal{P}$ must be~$1$ if~$R=\Fl$ and~$q\equiv 1[\ell]$, which proves our assertion in this case. In general,~$\mathcal{P}$ belongs to 
$\Zl[X]$, with constant term~$1$, as so do the polynomials~$1/L(X,\chi_i\mu_j)$. Let~$\mathcal{A}$ be the quotient~$\mathcal{P}/\mathcal{Q}$ in~$R[X]$. We have~$\mathcal{P}=\mathcal{A}\mathcal{Q}$ in~$R[X]$, with~$\mathcal{P}(0)=\mathcal{Q}(0)=1$, and~$\mathcal{P}\in \Zl[X]$. This equality in fact takes place in~$K[X]$, for~$K$ a finite extension of~$\Q_\ell$ such that~$\mathcal{Q}\in \o_K[X]$ (with~$\o_K$ the ring of integers in~$K$). As~$\o_K[X]$ is a unique factorisation domain, the fact that~$\mathcal{P}(0)=\mathcal{Q}(0)=\mathcal{A}(0)=1$ implies that~$\mathcal{A}$ and~$\mathcal{Q}$ are in fact in~$\Zl[X]$. 
\end{proof} 

\subsection{The functional equation}

We have defined Rankin--Selberg~$L$-factors of pairs of representations of Whittaker type, we now need to show that these satisfy a local functional equation.  By identifying~$F^n$ with~$\mathcal{M}_{1,n}$, the space~$\mathcal{C}_c^\infty (F^n)$ provides a smooth representation~$G_n$ by right translation, which we denote by~$\rho$. We also denote by~$\rho$ the action by right translation of~$G_n$ on any space of functions. For~$a\in R[X^{\pm1}]$, we denote by~$\chi_a$ the character in~$\Hom(G_n,R[X^{\pm 1}]^\times)$ defined by~$\chi_a:g\mapsto a^{v(det(g))}$.  In particular, we have~$\nu=\chi_{q^{1/2}}$ is the absolute value of the determinant.

Let~$\pi$ be an~$R$-representation of Whittaker type of~$G_n$, and~$\pi'$ be an~$R$-representation of Whittaker type of~$G_m$.  If~$m=n$, we write 
\[D(\pi,\pi',\mathcal{C}_c^\infty (F^n))=\Hom_{G_n}(\pi\otimes \pi'\otimes \mathcal{C}_c^\infty (F^n),\chi_X),\]  for the space of~$R$-linear maps,~$\mathfrak{L}:\pi\times \pi' \times \mathcal{C}_c^\infty (F^n)\rightarrow R[X^{\pm 1}]$, satisfying
\[\mathfrak{L}(\rho(h)W,\rho(h)W',\rho(h)\Phi)=X^k\mathfrak{L}(W,W',\Phi)\] 
for all~$W\in W(\pi,\theta)$,~$W'\in W(\pi',\theta^{-1})$,~$\Phi\in\mathcal{C}_c^\infty (F^n)$, and~$h\in G_n^k.$  If~$m\leqslant n-1$, we write 
 \[D(\pi,\pi')=\Hom_{G_{m}U_{m+1,n-m-1}}(\pi\otimes \pi',\chi_{q^{(n-m)/2}X}\otimes \theta),\] 
for the space of~$R$-linear maps,~$\mathfrak{L}:\pi\times \pi\rightarrow R[X^{\pm 1}]$, satisfying
\[\mathfrak{L}(\rho(h)W,\rho(h)W')=q^{k(n-m)/2}X^k\mathfrak{L}(W,W',\Phi), \quad \mathfrak{L}(\rho(u)W,W')=\theta(u)\mathfrak{L}(W,W')\]
for all~$W\in W(\pi,\theta)$,~$W'\in W(\pi',\theta^{-1})$,~$h\in G_m^k$, and~$u\in U_{m+1,n-m-1}$.  We denote by~$\mathcal{C}_{c,0}^\infty(F^n)$ the subspace of~$\mathcal{C}_{c}^\infty(F^n)$ which is the kernel of the evaluation map~$Ev_0:\Phi\mapsto \Phi(0)$.

\begin{prop}\label{mult1}
The spaces~$D(\pi,\pi',\mathcal{C}_c^\infty (F^n))$ and~$D(\pi,\pi')$ are free~$R[X^{\pm 1}]$-modules of rank~$1$. 
\end{prop} 

The proof in the complex case of Jacquet--Piatetskii-Shapiro--Shalika in \cite{JPS2} is long.  Some results obtained in the complex case [ibid.] using invariant distributions can be obtained quicker using derivatives which is how we proceed. 

\begin{proof}
We start with the case~$n=m$. The map \[(W,W',\Phi)\mapsto I(X,W,W',\Phi)/L(X,\pi,\pi')\] is a nonzero element of~$D(\pi,\pi',\mathcal{C}_c^\infty (F^n))$, hence we only need to show that~$D(\pi,\pi',\mathcal{C}_c^\infty (F^n))$ is a free~$R[X^{\pm 1}]$-module of rank at 
most~$1$.  

We have an exact sequence of representations of~$G_n$
\[0\rightarrow\mathcal{C}_{c,0}^\infty (F^n) \rightarrow \mathcal{C}_c^\infty (F^n) \rightarrow \1\rightarrow 0.\]
We tensor this sequence by~$\pi\otimes \pi'$ and, as~$\pi\otimes \pi'$ is flat as an~$R$-vector space, we obtain
\begin{equation*}\label{exact} 0\rightarrow \pi\otimes \pi' \otimes \mathcal{C}_{c,0}^\infty (F^n) \rightarrow \pi\otimes \pi'\otimes \mathcal{C}_c^\infty (F^n) \rightarrow \pi\otimes \pi'\rightarrow 0.\end{equation*}
By considering central characters, it is clear that the space~$\Hom_{G_n}(\pi\otimes \pi',\chi_X)=0$.  Applying 
$\Hom_{G_n}(~,\chi_X)$ which is left exact, we obtain that~$\Hom_{G_n}(\pi\otimes \pi' \otimes \mathcal{C}_c^\infty (F^n),\chi_X)$ is a~$R[X^{\pm 1}]$-submodule 
of~$\Hom_{G_n}(\pi\otimes \pi' \otimes \mathcal{C}_{c,0}^\infty (F^n),\chi_X)$. Hence it is sufficient to check that 
$\Hom_{G_n}(\pi\otimes \pi' \otimes \mathcal{C}_{c,0}^\infty (F^n),\chi_X)$ is of rank at most~$1$. 

Now, we have an isomorphism between~$\mathcal{C}_{c,0}^\infty (F^n)$ with~$G_n$ acting via right translation and~$\ind_{P_n}^{G_n}(\delta_{P_n}^{1/2})$, hence we have 
\begin{align*}
\Hom_{G_n}(\pi\otimes \pi' \otimes \mathcal{C}_{c,0}^\infty (F^n),\chi_X)&\simeq 
\Hom_{G_n}(\pi\otimes \pi',\chi_X \Ind_{P_n}^{G_n}(\delta_{P_n}^{-1/2})),\\
&\simeq \Hom_{P_n}(\pi\otimes \pi',\chi_X \delta_{P_n}^{-1})
\end{align*}
by Frobenius reciprocity. Now, by the theory of derivatives (see Subsection \ref{der}), 
$\pi$ and~$\pi'$, as~$P_n$-modules, are of finite length, with irreducible subquotients of the form
$(\Phi^+)^k\Psi^+ (\rho)$, for~$\rho$ an irreducible representation of~$G_{n-k-1}$ and~$k$ between~$0$ and~$n-1$. Moreover,~$(\Phi^+)^{n-1}\Psi^+ (\1)$ appears with multiplicity~$1$, as a submodule. By Proposition \ref{BZ3.7}, the space 
\[\Hom_{P_n}((\Phi^+)^k\Psi^+ (\rho) \otimes (\Phi^+)^j\Psi^+ (\rho'),\chi_X \delta_{P_n}^{-1})= 
\Hom_{P_n}((\Phi^+)^k\Psi^+ (\rho) \otimes (\Phi^+)^j\Psi^+ (\chi_X^{-1} \delta_{P_n}\rho') ,\1)\] is zero, 
except when~$j=k$, in which case it is isomorphic to~$\Hom_{G_k}( \rho \otimes \rho',\chi_X \nu^{-1}))$. If 
$\rho$ and~$\rho'$ are irreducible and~$k\geqslant 1$, by considering central characters, the space~$\Hom_{G_k}( \rho \otimes \rho',\chi_X \nu^{-1})$ 
is zero. Thus~$\Hom_{P_n}(\pi\otimes \pi',\chi_X \delta_{P_n}^{-1})$ is a~$R[X^{\pm 1}]$-submodule of 
$\Hom_{G_0}( \1 \otimes \1,\chi_X )\simeq R[X^{\pm 1}]$. This ends the proof in the case~$n=m$, as~$R[X^{\pm 1}]$ is principal.  

We now consider the case~$m\leqslant n-1$. Again, the space~$D(\pi,\pi')$ is nonzero as it contains the map~$(W,W')\mapsto I(X,W,W')/L(X,\pi,\pi')$, we will show that it injects into~$R[X^{\pm 1}]$, which will prove the statement. Let~$L$ be in~$D(\pi,\pi')$, by definition, the map~$L$-factors through~$\tau \times \pi'$, where~$\tau$ is the quotient of~$\pi$ by its subspace spanned by~$\pi(u)W-\theta(u) W$ for~$u\in U_{m+1,n-m-1}$ and~$W\in  W(\pi,\theta)$. Hence~$\tau$ is nothing other than the space of~$(\Phi^+)^{n-m-1}(\pi)$. Taking into account the normalisation in the definition of the derivatives, we obtain the following injection:
\[\Hom_{G_{m}U_{m+1,n-m-1}}(\pi\otimes \pi',\chi_{q^{(n-m)/2}X}\otimes \theta)\hookrightarrow \Hom_{G_m}((\Phi^-)^{n-m-1}(\pi) \otimes \pi',\chi_{q^{1/2}X}).\] We next prove the following lemma.
\begin{LM}\label{lminter}
If~$\sigma$ is an irreducible~$R$-representation of~$P_{m}$, then 
\[\Hom_{G_m}(\Phi^+(\sigma)\otimes \pi',\chi_{q^{1/2}X})\simeq \Hom_{P_m}(\sigma\otimes \pi',\chi_{X}).\]
If~$\sigma$ is an irreducible~$R$-representation of~$G_{m}$, then 
\[\Hom_{G_m}(\Psi^+(\sigma)\otimes \pi',\chi_{q^{1/2}X})\simeq \Hom_{G_m}(\sigma\otimes \pi',\chi_{X})=\{0\}.\] 
\end{LM}
\begin{proof}[Proof of the Lemma]
We first prove the second assertion. By definition of~$\Psi^+$, the~$R[X^{\pm 1}]$-module 
$\Hom_{G_m}(\Psi^+(\sigma)\otimes \pi',\chi_{q^{1/2}X})$ is equal to~$\Hom_{P_{m+1}}(\Psi^+(\sigma)\otimes \Psi^+(\pi'),\chi_{X}))$, which is itself 
isomorphic to~$\Hom_{G_m}(\sigma\otimes \pi',\chi_{X})$ by Proposition \ref{BZ3.7}. As~$\sigma$ is irreducible, it has a central character, and thus~$\Hom_{G_m}(\sigma,\chi_{X})=\{0\}$.

For the first assertion, as~$P_{m+1}=G_m (P_m U_{m+1})$, we have an isomorphism~$\Phi^+(\sigma)\mid_{G_m}\simeq \ind_{P_m}^{G_m}(\sigma)$ by Mackey theory (cf. \cite[Theorem 
5.2]{BZ}). Hence, we obtain
\[\Hom_{G_m}(\Psi^+(\sigma)\otimes \pi',\chi_{q^{1/2}X})\simeq 
\Hom_{G_m}(\ind_{P_m}^{G_m}(\sigma)\otimes \pi',\chi_{q^{1/2}X})\simeq \Hom_{P_m}(\sigma\otimes \pi',\chi_{X}),\] the last 
isomorphism by Frobenius reciprocity.
\end{proof}
Now,~$(\Phi^-)^{n-m-1}(\pi)$ is a~$P_{m+1}$-module of finite length, and as~$\pi$ is of Whittaker type, it contains~$(\Phi^+)^{m-1}\Psi^+(\1)$ as a submodule, the latter's multiplicity being~$1$ as a composition factor. By the theory of derivatives, all of the other irreducible subquotients are either of the form 
$\Psi^+(\sigma)$, with~$\sigma$ an irreducible representation of~$G_m$, or of the form~$\Phi^+(\sigma)$, with~$\sigma$ an irreducible representation of~$P_m$ 
of the form~$(\Phi^+)^{m-j-1}\Psi^+(\sigma')$, with~$\sigma'$ a representation of~$G_j$, for some~$j\geqslant 1$. By Lemma \ref{lminter},~$\Hom_{G_m}(\Psi^+(\sigma)\otimes \pi',\chi_{X})$ is zero. For all subquotients of the form~$\Phi^+(\sigma)$, we have~$\Hom_{G_m}(\Phi^+(\sigma)\otimes \pi',\chi_{q^{1/2}X})\simeq \Hom_{P_m}(\sigma\otimes \pi',\chi_{X})$ by Lemma \ref{lminter}.

\begin{LM}\label{lminter2}
The~$R[X^{\pm 1}]$-module~$\Hom_{P_m}(\sigma\otimes \pi',\chi_{X})$ is zero if~$\sigma$ is an irreducible~$R$-representation of~$P_m$ 
of the form~$(\Phi^+)^{m-j-1}\Psi^+(\sigma')$, with~$\sigma'$ an~$R$-representation of~$G_j$, for some~$j\geqslant 1$, whereas 
$\Hom_{P_m}((\Phi^+)^{m-1}\Psi^+(\1)\otimes \pi',\chi_X)\hookrightarrow R[X^{\pm 1}]$.
\end{LM}
\begin{proof}[Proof of the Lemma]
As~$\pi'$ is of Whittaker type, its restriction to~$P_m$ is of finite length, with irreducible subquotients of the form 
~$(\Phi^+)^{m-k-1}\Psi^+(\mu)$, for~$\mu$ an irreducible representation of~$G_k$. Moreover, the representation~$(\Phi^+)^{m-k-1}\Psi^+(\1)$ occurs 
 with multiplicity~$1$, and is a submodule.
If~$\sigma$ is an irreducible representation of~$P_m$ of the form~$(\Phi^+)^{m-j-1}\Psi^+(\sigma')$, with~$\sigma'$ a representation of~$G_j$, for some~$j\geqslant 1$, then~$\Hom_{P_m}((\Phi^+)^{m-h-1}\Psi^+(\sigma')\otimes (\Phi^+)^{m-k-1}\Psi^+(\mu),\chi_X)$ is zero by Proposition 
\ref{BZ3.7} if 
$j\neq k$ (in particular if~$k=1$). Moreover, if~$k=j$, by the same Proposition, we have~$\Hom_{P_m}((\Phi^+)^{m-j-1}\Psi^+(\sigma')\otimes (\Phi^+)^{m-k-1}\Psi^+(\mu),\chi_X)$ is isomorphic to~$\Hom_{G_j}(\sigma'\otimes \mu,\chi_X)$, which is zero, by considering central characters. Hence we have proved the first part of the lemma.  If~$\sigma=(\Phi^+)^{m-1}\Psi^+(\1)$, reasoning as above, we see at once that 
$\Hom_{P_m}((\Phi^+)^{m-1}\Psi^+(\1)\otimes \pi',\chi_X)$ injects into 
$\Hom_{P_m}((\Phi^+)^{m-1}\Psi^+(\1)\otimes (\Phi^+)^{m-1}\Psi^+(\1),\chi_X)\simeq \Hom_{G_0}(\1\otimes \1,\chi_X)$, the latter 
space being isomorphic to~$R[X^{\pm 1}]$, and this completes the proof of the lemma.
\end{proof}
All in all, we deduce that~$\Hom_{G_m}((\Phi^-)^{n-m-1}(\pi) \otimes \pi',\chi_{q^{1/2}X})$ injects as a~$R[X^{\pm 1}]$-submodule into 
$\Hom_{P_m}((\Phi^+)^{m-1}\Psi^+(\1)\otimes \pi',\chi_{X})$, which itself injects into~$R[X^{\pm 1}]$, and this ends the proof of the proposition.
\end{proof}

\begin{rem}
Notice that all the injections defined in the proof of Proposition \ref{mult1} are in fact isomorphisms. This could be viewed directly, or we can simply see that after composing all of them we obtain an isomorphism. 
\end{rem}

We are now in a position to state the local functional equation and define the Rankin--Selberg~$\varepsilon$-factor of a pair of representations of Whittaker type. We recall that an invertible element of~$R[X^{\pm 1}]$ is an element of the form~$cX^k$, for~$c$ in~$R^\times$, and~$k$ in~$\Z$.

\begin{cor}\label{fcteq}
Let~$\pi$ be an~$R$-representation of Whittaker type of~$G_n$, and~$\pi'$ be a representation of Whittaker type of~$G_m$. \begin{enumerate}
\item If~$m=n$, there is an invertible element~$\varepsilon(X,\pi,\pi',\theta)$ of the ring~$R[X^{\pm 1}]$ such that for any~$W\in W(\pi,\theta)$, any~$W'\in W(\pi',\theta^{-1})$, and any~$\Phi$ 
in~$\mathcal{C}_c^\infty(F^n)$, we have:
\[\frac{I(q^{-1}X^{-1},\widetilde{W},\widetilde{W'},\widehat{\Phi})}{L(q^{-1}X^{-1},\widetilde{\pi},\widetilde{\pi'})}= c_{\pi'}(-1)^{m-1}\varepsilon(X,\pi,\pi',\theta)\frac{I(X,W,W',\Phi)}{L(X,\pi,\pi')}.\]
\item If~$m\leqslant n-1$, there is an invertible element~$\varepsilon(X,\pi,\pi',\theta)$ of the ring~$R[X^{\pm 1}]$ such that, for any~$W\in W(\pi,\theta)$, any~$W'\in W(\pi',\theta^{-1})$, and any~$0\leqslant j \leqslant n-m-1$, we have:
\[\frac{I(q^{-1}X^{-1},\rho(w_{m,n-m})\widetilde{W},\widetilde{W'};n-m-1-j)}{L(q^{-1}X^{-1},\widetilde{\pi},\widetilde{\pi'})}= 
c_{\pi'}(-1)^{m-1}\varepsilon(X,\pi,\pi',\theta)
\frac{I(X,W,W';j)}{L(X,\pi,\pi')}.\]
\end{enumerate}
\end{cor}
\begin{proof}
It is a consequence of Proposition \ref{mult1} if~$n=m$, and if~$m\leqslant n-1$ with~$j=0$, as the functionals on both sides of the equality belong, respectively, to~$D(\pi,\pi',\mathcal{C}_c^\infty (F^n))$ and~$D(\pi,\pi')$. For~$j\neq 0$, it follows from the case~$j=0$ as in the complex setting, cf. \cite{JPS1}.\end{proof}

We call~$\varepsilon(X,\pi,\pi',\theta)$ the \emph{local~$\varepsilon$-factor} associated to~$\pi,\pi'$, and~$\theta$, and we write 
\[\gamma(X,\pi,\pi',\theta)=\frac{\varepsilon(X,\pi,\pi',\theta)L(q^{-1}X^{-1},\widetilde{\pi},\widetilde{\pi'})}{L(X,\pi,\pi')},\]
for the \emph{local~$\gamma$-factor} associated to~$\pi,\pi'$, and~$\theta$.

\subsection{Compatibility with reduction modulo~$\ell$}

Let~$\pi$ and~$\pi'$ are integral~$\ell$-adic representations of Whittaker type of~$G_n$ and~$G_m$. By Corollary \ref{div}, we already know that~$L(X,\pi,\pi')$ is the inverse of a polynomial with integral coefficients, even without the integrality assumption. With the integrality assumption, we now consider the associated~$\varepsilon$-factor.

\begin{LM}\label{epsilon}
The factor~$\varepsilon(X,\pi,\pi',\theta)$ is of the form~$cX^k$, for~$c$ a unit in~$\Zl$.
\end{LM}
\begin{proof}
We only do the case~$n=m$, the case~$m\leqslant n-1$ follows \emph{mutatis mutandis}. By Remark \ref{intvalues}, whenever~$W$,~$W'$ and~$\Phi$ have integral values, the Laurent series~$I(X,W,W',\Phi)$ and~$I(q^{-1}X^{-1},\widetilde{W},\widetilde{W'},\hat{\Phi})$ belong, respectively, to~$\Zl((X))$ and~$\Zl((X^{-1}))$. As the factors~$L(X,\pi,\pi')$ and~$L(q^{-1}X^{-1},\widetilde{\pi},\widetilde{\pi'})$ are the inverse of polynomials 
 in~$\Zl[X^{\pm1}]$ with constant term~$1$, the quotient 
\[I(q^{-1}X^{-1},\widetilde{W},\widetilde{W'},\hat{\Phi})/L(q^{-1}X^{-1},\widetilde{\pi},\widetilde{\pi'})\] which belongs to~$R[X^{\pm 1}]$, in fact belongs to~$R[X^{\pm 1}]\cap\Zl((X^{-1}))=\Zl[X^{\pm 1}],$ and simililarly for the quotient 
~$I(X,W,W',\Phi)/L(X,\pi,\pi')$. The functional equation then implies that the scalar~$c$ is in~$\Zl$, and applying it twice
  shows that it is in~$\Zl^\times$.
\end{proof}

If~$P$ is an element of~$\Zl[X]$ with nonzero reduction modulo~$\ell$, we write~$r_{\ell}(P^{-1})$ for~$(r_{\ell}(P))^{-1}$. 
We now prove our first main result.

Let~$\pi$ and~$\pi'$ be~$\ell$-modular representations of Whittaker type of~$G_n$ and~$G_m$.  Let~$\tau$ and~$\tau'$ be Whittaker lifts of~$\pi$ and~$\pi'$ (respectively).

\begin{thm}\label{compat1}
We have
\[L(X,\pi,\pi')\mid r_{\ell}(L(X,\tau,\tau')),\] and 
\[\gamma(X,\pi,\pi',r_{\ell}(\theta))=r_{\ell}(\gamma(X,\tau,\tau',\theta)).\]
\end{thm}
\begin{proof}
We give the proof for~$m\leqslant n-1$, and~$j=0$, the other cases being similar. By definition, one can write~$L(X,\pi,\pi')$ as a finite sum 
$\sum_i I(X,W_i,W_i')$, for~$W_i \in W(\pi,\theta)$ and~$W_i' \in W(\pi',\theta^{-1})$. By Lemma \ref{dual-lattice}, there are Whittaker functions~$W_{i,e}\in W_e(\tau,\mu)$ and~$W_{i,e}'\in W_e(\tau',\mu^{-1})$ such that~$\widetilde{W_{i,e}}\in W_e(\widetilde{\tau},\mu^{-1})$ and~$\widetilde{W_{i,e}'}\in W_e(\widetilde{\tau'},\mu)$ , and such that~$W_i=r_{\ell}(W_{i,e})$, and~$W_i'=r_{\ell}(W_{i,e}')$. By Remark \ref{intvalues}, we have~$L(X,\pi,\pi')=r_{\ell}(\sum_i I(X,W_{i,e},W_{i,e}'))$. As the sum 
$\sum_i I(X,W_{i,e},W_{i,e}')$ belongs to 
\[L(X,\tau,\tau')\Ql[X^{\pm 1}]\cap \Zl((X))=L(X,\tau,\tau')\Zl[X^{\pm 1}],\] we obtain that 
$L(X,\pi,\pi')$ belongs to~$r_{\ell}(L(X,\tau,\tau'))\Fl[X^{\pm 1}]$.  This proves the first assertion.  The equality for~$\gamma$-factors follows the functional equation, and Remark \ref{intvalues}.  
\end{proof}

As an immediate corollary of Theorem \ref{compat1}, of \cite[Proposition 2.13]{JPS2} and of \cite[Proposition 4.1, Remark 5.2]{JSHighly}, we obtain the stability of the local factors. We recall that the level~$l(\chi)$ of a character~$\chi$ of~$F^\times$ is either zero if it is trivial on~$\o^\times$, or the smallest integer~$n\geqslant 1$ such that~$\chi$ trivial on~$1+\p^n$. 

\begin{cor}\label{stability}
Let~$\pi$ and~$\pi'$ be a pair of~$\ell$-modular representations of Whittaker type of~$G_n$ and~$G_m$ respectively, then there exists~$l_{\pi,\pi'}\in \N$ such that
\[L(X,\pi\otimes \chi,\pi')=1,\]
 whenever~$l(\chi)\geqslant l_{\pi,\pi'}$. 
 Moreover, if~$\pi_1,\pi_2$ are~$\ell$-modular representations of Whittaker type of~$G_n$ with equal central characters, there exists~$l_{\pi'}\in\N$ such that
 \[\gamma(X,\pi_1\otimes\chi,\pi',\theta)=\gamma(X,\pi_2\otimes\chi,\pi',\theta),\]
whenever~$l(\chi)\geqslant l_{\pi'}$. 
\end{cor}

\begin{rem}
As for the Godement--Jacquet~$L$-functions (see \cite{M}), we do not always have an equality between~$r_{\ell}(L(X,\tau,\tau'))$ and~$L(X,\pi,\pi')$. For instance when~$q\equiv 1[\ell]$, we always have~$L(X,\pi,\pi')=1$. But already, for unramified~$\ell$-adic characters~$\tau$ and~$\tau'$ of~$G_1$, we have \[r_{\ell}(L(X,\tau,\tau'))=1/(1-X).\] 
This will become completely transparent in the generic case by the end of the next section.
\end{rem}

\section{The inductivity relation, and explicit computations.}

In Section \ref{gamma-section}, we obtain the inductivity relation of~$\gamma$-factors of representations of Whittaker type by reduction modulo~$\ell$.  For~$L$-factors, we restrict to pairs of generic representations and obtain again the inductivity relation, and an explicit formula, differing from the~$\ell$-adic 
case due to the presence of non-banal representations. In fact, for generic representations we show that the~$L$-factor only depends on the banal parts of the representations in Theorem \ref{banal-part-factors}.  We also discuss precisely the relation with reduction modulo~$\ell$, and obtain in Theorem \ref{GCDtheorem} a nice interpretation of the~$\ell$-modular~$L$-factor of a pair of generic representations as the greatest common divisor of the~$L$-factors of the pairs of generic~$\ell$-adic representations containing our generic pair as subquotients on reduction modulo~$\ell$. This also provides a coherent way of defining~$L$-factors of pairs of irreducible~$\ell$-modular representations.  

\subsection{Gamma factors of representations of Whittaker type}\label{gamma-section}
%
 
Let~$R$ be either~$\Ql$ or~$\Fl$. Let~$\pi_1,\pi_2$ be~$\ell$-modular representations of Whittaker type and~$\tau_1,\tau_2$ be standard lifts of~$\pi_1,\pi_2$ respectively.  For~$i=1,2$, let~$\tau_i'$ be a subquotient of~$\tau_i$ of Whittaker type, and~$\pi_i'$ be a subquotient of~$\pi_i$ of Whittaker type.  As the functor
\[\pi\mapsto \Hom_{N_n}(\pi,\theta)\] 
is exact from~$\mathfrak{R}_R(G_n)$ to the category of~$R$-vector spaces, by multiplicity one of the Whittaker functional for Whittaker representations, we deduce that 
the (unique up to scaling) Whittaker functional on~$\pi_1'$ is the restriction of that on~$\pi_1$ (and similarly for the pairs~$(\pi_2',\pi_2)$,~$(\tau_1',\tau_1)$, and~$(\tau_2',\tau_2)$). 
%

This implies immediately the following inclusions of Whittaker models:
\begin{align*}
W(\tau_1',\theta)&\subseteq W(\tau_1,\theta),&W(\tau_2',\theta^{-1})&\subseteq W(\tau_2,\theta^{-1}),\\
W(\pi_1',r_\ell(\theta))&\subseteq W(\pi_1,r_\ell(\theta)), &W(\pi_2',r_\ell(\theta^{-1}))&\subseteq W(\pi_2,r_\ell(\theta^{-1})). 
\end{align*}  These inclusions being equalities whenever the respective subquotients are actually quotients. 

Applying~$W\mapsto \widetilde{W}$ (in~$\mathcal{C}^\infty(N_{n_i}\backslash G_{n_i})$) to the above inclusions, we obtain the inclusions:
\begin{align*}
W(\widetilde{\tau_1}',\theta^{-1})&\subseteq W(\widetilde{\tau_1},\theta^{-1}),&W(\widetilde{\tau_2}',\theta)
&\subseteq W(\widetilde{\tau_2},\theta),\\
W(\widetilde{\pi_1}',r_\ell(\theta^{-1}))&\subseteq W(\widetilde{\pi_1},r_\ell(\theta^{-1})), &W(\widetilde{\pi_2}',r_\ell(\theta))&\subseteq W(\widetilde{\pi_2},r_\ell(\theta)). 
\end{align*}

Thanks to all of the above inclusions of Whittaker models, the functional equation gives the equalities of gamma factors:
\begin{align*}
\gamma(X,\pi_1',\pi_2')&=\gamma(X,\pi_1,\pi_2),&\gamma(X,\tau_1',\tau_2')&=\gamma(X,\tau_1,\tau_2).
\end{align*}

 Now, by Theorem \ref{compat1}, we have 
 \[r_\ell(\gamma(X,\tau_1,\tau_2))=\gamma(X,\pi_1,\pi_2),\] 
 hence 
\[r_\ell(\gamma(X,\tau_1',\tau_2'))=\gamma(X,\pi_1',\pi_2').\]

In particular, we can take~$\tau_i'$ (resp.~$\pi_i'$) to be the unique generic subquotient of~$\tau_i$ (resp.~$\pi_i$) for~$i=1,2$.

The following theorem follows by reduction modulo~$\ell$ (Theorem \ref{compat1}) of the inductivity relation of~$\ell$-adic gamma factors in \cite[3.1]{JPS2}. 

\begin{thm}\label{inductivity-gamma}
Let~$\pi_1$ and~$\pi_2$ and~$\pi_3$ be~$\ell$-modular representations of Whittaker type, of~$G_{n_1}$,~$G_{n_2}$ and 
$G_{n_3}$ respectively. Then 
\[\gamma(X,\pi_1,\pi_2\times\pi_3)=\gamma(X,\pi_1,\pi_2)\gamma(X,\pi_1,\pi_3).\]
\end{thm}

Together with the discussion above, it has the following corollary:

\begin{cor}\label{gamma-steinberg}
Let~$\rho$ and~$\rho'$ be cuspidal~$\ell$-modular representations of 
$G_r$ and~$G_{r'}$ respectively, and~$a\leqslant b$ and~$c\leqslant d$ in~$\Z$. 
Then one has~\[\gamma(X,St(\rho,[a,b]),St(\rho',[c,d]))=\prod_{i=a}^b \prod_{j=c}^d \gamma(X,\nu^i\rho,\nu^j\rho').\]
\end{cor}
\begin{proof}
As~$St(\rho,[a,b])$ and~$St(\rho,[c,d])$ are the unique generic quotients of the induced representations 
$\nu^a\rho\times \dots \times \nu^b \rho$ and~$\nu^c\rho'\times \dots \times \nu^d\rho'$ respectively, the assertion 
follows from Theorem \ref{inductivity-gamma} and the discussion preceding it.
\end{proof}

\subsection{Divisibility relations between~$L$-factors of induced representations}

%
In this subsection~$R$ can be~$\Ql$ or~$\Fl$. We first start with a simple observation, which follows from the inclusions between Whittaker models discussed in the beginning of Section 
\ref{gamma-section}.

\begin{LM}\label{div-subquotient}
Let~$\pi$ (resp.~$\tau$) be an~$R$-representation of Whittaker type of~$G_n$ (resp.~$G_m$), and~$\pi'$ (resp.~$\tau'$) be a subquotient of Whittaker type (for example the unique generic subquotient) of~$\pi$ 
(resp.~$\tau$), then~$L(X,\pi',\tau')$ divides~$L(X,\pi,\tau)$, and it is equal to it if both~$\pi'$ and~$\tau'$ are quotients of 
$\pi$ and~$\tau$.
\end{LM}

We also observe, as in \cite[Section 7]{JPS2} (their proof being valid for~$R=\Fl$), thanks to the stability of~$L$-factors under highly ramified twists (Corollary \ref{stability}), that the inductivity relation of~$\gamma$-factors (Theorem \ref{inductivity-gamma}) implies the following result.

\begin{LM}\label{Lintermediate}
Let~$\pi_1$,~$\pi_2$ and~$\pi_3$ be~$R$-representations of Whittaker type of~$G_{n_1}$,~$G_{n_2}$ and~$G_{n_3}$ respectively, then~$L(X,\pi_1\times \pi_2,\pi_3)$ divides~$L(X,\pi_1,\pi_3)L(X,\pi_2,\pi_3)$.
\end{LM}

We now recall \cite[Proposition 9.1]{JPS2}. A small part of the proof of [ibid.] is given in a particular (but in fact very general) case, we give the very slight changes here, to obtain the general proof. We shall also give some useful corollaries.

\begin{LM}\label{extension-of-whittaker-functions}
Let~$\pi_1$ and~$\pi_2$ be two~$R$-representations of Whittaker type of~$G_{n_1}$ and~$G_{n_2}$ respectively, and let~$\pi=\pi_1\times \pi_2$. Then for any pair~$(W_2,\Phi)\in W(\pi_2,\theta)\times \mathcal{C}_c^\infty(F^{n_2},R)$, there is~$W\in W(\pi,\theta)$ such that for~$g\in G_{n_2}$, we have
\[W(\diag(g,I_{n_1}))=W_2(g)\Phi(\eta g)\nu(g)^{n_1/2}.\]
\end{LM}
\begin{proof}
Because $W(\pi_1\times \pi_2,\theta)$ is queal to $W(W(\pi_1,\theta)\times W(\pi_2,\theta),\theta)$, we can replace 
$\pi$ by $\pi'=W(\pi_1,\theta)\times W(\pi_2,\theta)$ to prove the statement. We write~$P$ for~$P_{(n_1,n_2)}$, and~$U$ for its unipotent radical. The group~$U$ is isomorphic to~$\mathcal{M}=\mathcal{M}(n_1,n_2,F)$ via the map 
\[u:x\mapsto \begin{pmatrix} I_{n_1} & x \\  & I_{n_2} \end{pmatrix}.\] We can 
consider~$W(\pi_1,\theta)\otimes W(\pi_2,\theta)$ as a~$P$-module (with~$U$ acting trivially). We first recall the following standard fact on induced representations of locally profinite groups: the linear map~$I:\mathcal{C}_c^\infty(G_n,R)\otimes W(\pi_1,\theta)\otimes W(\pi_2,\theta)\rightarrow \pi'$, defined by the formula 
\[I(F\otimes W_1 \otimes W_2)(g)=\int_{P} F(pg) \d_P^{-1/2}(p)\rho(p)^{-1} W_1\otimes W_2 dp,\]
for~$F\in\mathcal{C}_c^\infty(G_n,R)$,~$W_1\in W(\pi_1,\theta)$ and~$W_2\in W(\pi_2,\theta)$, is a well defined surjection.  Taking~$F$ with support in~$PwU$, of the form~$F(pwu(x))=\beta(p)\Psi(x)$ with~$\beta$ the characteristic function of a small enough compact open subgroup of~$P$ (so that this subgroup has a nonzero volume, and fixes~$W_1\otimes W_2$), and with~$\Psi$ any function in
~$\mathcal{C}_c(\mathcal{M},R)$, we obtain the map~$f=I(F)$ in~$\pi'$, which satisfies, up to nonzero scaling the relation:
\[f(w u(x))=\Psi(x) W_1\otimes W_2.\] Now the end of the proof is that of \cite[Proposition 9.1]{JPS2}.
\end{proof}

The following result also follows from the results of \cite{JPS2}.

\begin{LM}\label{inclusion-between-ideals}
Let~$m,n$ be positive integers, with~$m \leqslant n$. If~$\pi$ and~$\pi'$ are two~$R$-representations of~$G_n$ of Whittaker type, and if 
$W$ and~$W'$ belong respectively to~$W(\pi,\theta)$ and~$W(\pi',\theta^{-1})$, the integral 
\[I_{(n-m-1)}(X,W,W')=\sum_{k\in \Z}\left(\int_{N_n\backslash G_n^k} W(\diag(g,I_{n-m}))W'(\diag(g,I_{n-m})) dg\right)q^{k(n-m)}X^k\]
belongs to the fractional ideal~$I(\pi,\pi')$ (see Theorem \ref{Lfactorsdefined}).
\end{LM}
\begin{proof}
Thanks to \cite[Lemma 9.2]{JPS2}, the proof of which is valid over~$R$, we see that there are~$W_0\in W(\pi,\theta)$ and~$W_0'\in W(\pi',\theta^{-1})$ such that~$I_{(n-m-1)}(X,W,W')=I_{(0)}(X,W_0,W_0')$. Hence it suffices to prove the lemma when~$m=n-1$. In this case, for any~$\Phi$ in~$\mathcal{C}_c^\infty(F^n)$, we denote by~$f_{\Phi}$ the function on~$F$ defined as~$f_{\Phi}:t\mapsto \Phi(0,\dots,0,t)$. Using the integration formula of Corollary \ref{iwasplit2}, we can write 
\begin{equation*}\label{L0Lradex}
I(X,W,W',\Phi)=\int_{(K_n\cap P_n)\backslash K_n}I_{(0)}(X,\rho(k)W,\rho(k)W')
I(X^n,c_\pi,c_{\pi'},f_{\rho(k)\Phi})dk,
\end{equation*}
where~$\rho$ denotes right translation. Writing~$K_{n,r}=1+\w^r\mathcal{M}(n,\o)$, and taking~$\Phi=\1_{K_{n,r}}$ for~$r$ large enough for~$W$ and~$W'$ to be invariant under~$K_{n,r}$, our integral reduces to a nonzero multiple of~$I_{(0)}(X,W,W')$, and the result follows.
\end{proof}

We now deduce from Lemmas \ref{extension-of-whittaker-functions} and \ref{inclusion-between-ideals}, the following proposition, for the proof of which we refer to \cite[Proposition 4.4]{CP}.

\begin{prop}\label{div-induced}
Let~$\pi_1$ and~$\pi_2$ be two~$R$-representations of Whittaker type of~$G_{n_1}$ 
and~$G_{n_2}$ respectively, let~$\pi=\pi_1\times \pi_2$, and put~$n=n_1+n_2$.  Let~$\pi'$ be a representation of Whittaker type of~$G_n$. 
Then~$L(X,\pi_2,\pi')^{-1}$ divides~$L(X,\pi,\pi')^{-1}$ in~$R[X]$. 
\end{prop}

It has the following corollary, which we will use several times.

\begin{cor}\label{important-cor}
Let~$\rho$ be a cuspidal~$R$-representation of~$G_m$ ($m\geqslant 1$), and~$a< b\leqslant c$ two positive integers such that~$c-a+1< e(\rho)$, so that~$St(\rho,[a,b-1])$,~$St(\rho,[b,c])$ and~$St(\rho,[a,c])$ are generic segments. Then, for any~$\ell$-modular representations~$\pi$ and~$\pi'$ of Whittaker type of~$G_n$ and~$G_{n'}$ respectively, we have
\[L(X,St(\rho,[b,c])\times \pi,\pi')\mid L(X,St(\rho,[a,c])\times \pi,\pi').\]
\end{cor}
\begin{proof}
Because~$St(\rho,[a,c])\times \pi$ is the generic quotient (remember that~$c-a+1< e(\rho)$) of~$St(\rho,[a,b-1])\times St(\rho,[b,c])\times \pi$, Lemma \ref{div-subquotient} gives us the relation 
\[L(X,St(\rho,[a,c])\times \pi,\pi')= L(X,St(\rho,[a,b-1])\times St(\rho,[b,c])\times \pi,\pi').\] 
The result now follows from Proposition \ref{div-induced}.
\end{proof}

\subsection{$L$-factors of cuspidal representations}

We are now going to give the classification of all cuspidal~$\ell$-modular~$L$-factors.  We thus recall that according to  \cite[Theorem 6.14]{MS} (cf. \cite[5.14]{V}), any cuspidal representation~$\pi$ of~$G_n$ is of the form~$\pi=St_r(\rho)=St(\rho,e(\rho)\ell^r)$ for~$\rho$ a supercuspidal representation. In particular a banal cuspidal representation of~$G_n$ is always supercuspidal.

\begin{thm}\label{cuspidal}
Let~$\pi_1$ and~$\pi_2$ be two cuspidal~$\ell$-modular (or~$\ell$-adic) representations of~$G_{n_1}$ and~$G_{n_2}$. Then 
$L(X,\pi_1,\pi_2)$ is equal to~$1$, except in the following case:~$\pi_1$ is banal (hence supercuspidal), and~$\pi_2\simeq \chi\pi_1^\vee$ for some unramified character~$\chi$ of~$F^\times~$ (in particular~$n_1=n_2$). In this case, let~$f=f(\pi_1)=f(\pi_2)$, we have 
\[L(X,\pi_1,\pi_2)=\frac{1}{1-(\chi(\w_F)X)^{f}}\]
and this factor is the reduction modulo~$\ell$ of the~$L$-factor of any cuspidal lifts of~$\pi_1$ and~$\pi_2$.
\end{thm}

\begin{proof}
The banal case is the main application to reduction modulo~$\ell$ of \cite{KM15}.  Thus, it remains to consider the case where one of the representations is non-banal. By choosing lifts of~$\pi_1$ and~$\pi_2$, by Theorem \ref{compat1}, we are done if~$\pi_2\not\simeq \chi\pi_1^\vee$, for some unramified character~$\chi$ of~$F^\times$.  Assume first that~$\pi_2$ is non-banal and supercuspidal. Thus, by Proposition \ref{supercuspliftingnoniso} there are cuspidal lifts~$\rho_2,\rho_2'$ of~$\pi_2$ which are not isomorphic by twisting by an unramified character.  In particular, if we let~$\rho_1$ be a cuspidal lift of~$\pi_1$, one of the~$L$-factors~$L(X,\rho_1,\rho_2)$ or~$L(X,\rho_1,\rho_2')$ is trivial. But as~$L(X,\pi_1,\pi_2)$ divides the reduction of both of these factors, the result follows.  

Thus we can assume that~$\pi_2$ is cuspidal non-supercuspidal and, as~$\pi_2\simeq \chi\pi_1^\vee$ for some unramified character~$\chi$ of~$F^\times$, both of~$\pi_1$ and~$\pi_2$ are cuspidal non-supercuspidal.  We write~$\pi_1=St_{r_1}(\rho_1)$ and~$\pi_2=St_{r_2}(\rho_2)$.  We assume that~$n_1=m_1 e(\rho_1)\ell^{r_1}\geq n_2=m_2 e(\rho_2)\ell^{r_2}$, where~$\rho_1$ is a supercuspidal representation of~$G_{m_1}$ and~$\rho_2$ is a supercuspidal representation of~$G_{m_2}$. 
Moreover, as neither of~$\St_{r_1}(\rho_1)$ or~$\St_{r_2}(\rho_2)$ is supercuspidal, we have~$e(\rho_1)\ell^{r_1}>1$ and~$e(\rho_2)\ell^{r_2}>1$ (and, in fact, we only really need to assume one of the two for the remainder of the proof). In this case, we have
 \[L(X,St_{r_1}(\rho_1),St_{r_2}(\rho_2))\mid L(X,St_{r_1}(\rho_1),\rho_2\times \dots \times \nutwo^{e(\rho_2)\ell^{r_2}-1}\rho_2)\] because 
$St_{r_2}(\rho_2)$ being is the generic subquotient of \[\rho_2\times \dots \times \nutwo^{e(\rho_2)\ell^{r_2}-1}\rho_2,\] by Lemmas
 \ref{div-subquotient} and \ref{Lintermediate}, 
we deduce that \[L(X,St_{r_1}(\rho_1),St_{r_2}(\rho_2))\mid \prod_{i=0}^{e(\rho_2)l^{r_2}-1}L(X,St_{r_1}(\rho_1),\nutwo^i\rho_2)).\]
But each factor~$L(X,St_{r_1}(\rho_1),\nutwo^i\rho_2))$ is equal to~$1$ because~$n_1>m_2$, hence the result.
\end{proof}

\subsection{$L$-factors of generic segments}

By the definition of the ramification index and the lifting of cuspidal~$\ell$-modular representations recalled in Section \ref{liftingsection} we have the following proposition, which will allow us to study reduction modulo~$\ell$ of banal~$L$-factors of cuspidal representations.

\begin{prop}\label{cuspidal-preservation}
Let~$\tau$ be an~$\ell$-adic cuspidal representation of~$G_n$ and~$\rho$ be the reduction modulo~$\ell$ of~$\tau$, then~$q(\tau)=q(\rho)$ and~$f(\tau)=f(\rho)$. 
\end{prop}

We now move on to the case of generic segments. We notice that if~$\rho$ is a cuspidal representation 
of~$G_r$, the condition~$k<e(\rho)$ implies that~$St(\rho,k)$ is non-banal if and only if~$\rho$ is non-banal.  Let~$D_1=St(\tau_1,k_1)$ and~$D_2=St(\tau_2,k_2)$ be generic segments of~$G_{n_1}$ and~$G_{n_2}$, with~$n_1\geqslant n_2$. When~$R=\Ql$, the following formula was proved in \cite[Theorem 8.2]{JPS2}:

\begin{equation}\label{complex-L-segments} L(X,D_1,D_2)=\prod_{j=0}^{k_2-1} L(X,\nutwo^{k_1-1}\tau_1,\nutwo^j\tau_2).\end{equation}

As a consequence, we obtain the following useful Lemma. 

\begin{LM}\label{poles-segments}
Let~$\rho$ be a cuspidal~$\ell$-modular representation, and~$St(\rho,k_1)$ and~$St(\rho^\vee,k_2)$ be generic segments, with~$k_1\geqslant k_2$. Write~$P$ the set of poles of~$L(X,St(\rho,k_1),St(\rho^\vee,k_2)).$ Then~$P^{f(\rho)}=\{x^{f(\rho)},\ x\in P\}$ is a subset of~$\{q(\rho)^{k_1-1},\dots,q(\rho)^{k_1+k_2-2}\}$.
\end{LM}
\begin{proof}
Let~$\tau$ be a cuspidal~$\ell$-adic lift of~$\rho$, so in particular~$f(\rho)=f(\tau)$ and~$q(\tau)=q(\rho)$ according to 
Proposition \ref{cuspidal-preservation}. We recall that
~$L(X,St(\rho,k_1),St(\rho^\vee,k_2))$ divides~$r_\ell(L(X,St(\tau,k_1),St(\tau^\vee,k_2)))$ according to Proposition \ref{compat1}, and 
this last factor is equal to~$\prod_{j=0}^{k_2-1}r_\ell(L(q^{-k_1-j-1}X,\tau,\tau^\vee))$, according to Equation (\ref{complex-L-segments}). Now, thanks to Theorem \ref{cuspidal}, we have
\[r_\ell(L(q^{-k_1-j-1}X,\tau,\tau^\vee))=\frac{1}{1-q(\rho)^{-k_1-j-1}X^f},\]
and the result follows.
\end{proof}                                                       

%
We now completely describe the~$\ell$-modular~$L$-factors for generic segments.

\begin{thm}\label{inductivity-segments}
Let~$\L(\D_1)=St(\rho_1,k_1)$ and~$\L(\D_2)=St(\rho_2,k_2)$ be two~$\ell$-modular generic segments of~$G_{n_1}$ and~$G_{n_2}$ respectively, 
with~$n_1\geqslant n_2$, then:
\begin{enumerate}
\item If~$\L(\D_1)$ or~$\L(\D_2)$ is non-banal, then \[L(X,\L(\D_1),\L(\D_2))=1.\]
\item\label{part2inductivity-segments} If~$\L(\D_1)$ and~$\L(\D_2)$ are both banal, then 
\[L(X,\L(\D_1),\L(\D_2))=\prod_{i=0}^{k_2-1} L(X,\nutwo^{k_1-1}\rho_1,\nutwo^i\rho_2).\] 
In this case, if~$\tau_1$ and~$\tau_2$ are cuspidal lifts of~$\rho_1$ and~$\rho_2$, so that~$\L(D_1)=St(\tau_1,k_1)$ and~$\L(D_2)=St(\tau_2,k_2)$ are generic segments lifting~$\L(\D_1)$ and~$\L(\D_2)$, then
\[L(X,\L(\D_1),\L(\D_2))=r_\ell(L(X,\L(D_1),\L(D_2)).\]
\end{enumerate}
\end{thm}
\begin{proof}
First we notice that the last formula concerning reduction modulo~$\ell$ follows from Part \ref{part2inductivity-segments}, Equation (\ref{complex-L-segments}) and the last part of Theorem \ref{cuspidal}. We also notice that according 
to Proposition \ref{banal-part}, the segment~$\L(\D_i)$ is banal if and only if~$\rho_i$ is banal, for~$i=1,2$. In case either case, according to Lemmas \ref{div-subquotient} and \ref{Lintermediate}, we have 
\[L(X,\L(\D_1),\L(\D_2))\mid \prod_{i,j} L(X,\nutwo^{i}\rho_1,\nutwo^{j}\rho_2).\] 
If either segment is non-banal, each factor~$L(X,\nutwo^{i}\rho_1,\nutwo^{j}\rho_2)$ is equal to~$1$ according to Theorem \ref{cuspidal}, and the first assertion follows. In fact, the same argument in the non-banal case shows that if~$\rho_1^\vee$ and~$\rho_2$  are not isomorphic up to twisting by an unramified character, then~$L(X,\D_1,\D_2)=1$.

Hence we are left with Part \ref{part2inductivity-segments}, and moreover we can assume~$\rho_2=\chi \rho_1^\vee$ for some unramified character~$\chi$ of~$F^\times$. Thanks to the relation~$L(X,\pi,\chi\pi')=L(\chi(w_F)X,\pi,\pi')$ for any representations~$\pi$ and~$\pi'$ of Whittaker type, it is enough to prove the result when~$\rho_2=\rho_1^\vee$. We thus set~$\rho=\rho_1$ and~$\rho^\vee=\rho_2=\rho_1^\vee$. From now on, as we are dealing with banal representations, it is possible to follow the main idea of the proof of \cite{JPS2} up to requested adjustments (as some of the arguments in this reference depend on poles being in right and left half planes of~$\C$). In particular we proceed by induction on~$k_2$. We recall that because~$\L(\D_1)$ and~$\L(\D_2)$ are generic, we have the inequalities~$k_i< e(\rho)$ for~$i=1,2$, and because they are banal, we have~$e(\rho)=o(\rho)\geqslant 2$.

If~$k_2=1$, then~$k_1\geqslant 1$. In this case, we proceed by induction on~$k_1$. If~$k_1=1$, our assertion follows from Theorem \ref{cuspidal}. If~$k_1>1$, as~$St(\rho,k_1)$ is the generic subquotient of~$\rho \times St(\nutwo\rho,k_1-1)$, we have
\[L(X,St(\rho,k_1),\rho^\vee)\mid L(X,St(\nutwo\rho,k_1-1),\rho^\vee)L(X,\rho,\rho^\vee)\]
 by Lemmas \ref{div-subquotient} and Proposition \ref{Lintermediate}. Hence there is~$R\in \Fl[X^{\pm 1}]$, such that 
\begin{equation} 
L(X,St(\rho,k_1),\rho^\vee)=R(X)L(X,St(\nutwo\rho,k_1-1),\rho^\vee)L(X,\rho,\rho^\vee).
\end{equation}
We let~$e=e(\rho)=o(\rho)$,~$f=f(\rho)$,~$P_1$ be the set of poles of~$L(X,\rho,\rho^\vee)$, and~$P_2$ be the set of poles of~$L(X,St(\rho,k_1),\rho^\vee)$. By Lemma \ref{poles-segments}, one has~$P_1^f\subseteq \{1\}$, whereas
\[P_2^f\subseteq \{q(\rho)^{k_1}\}\subseteq \{q(\rho),\dots,q(\rho)^{e-1}\},\] 
and, in particular,~$P_1^{f}\cap P_2^{f}$, hence~$P_1\cap P_2$, is empty.  Thus~$P(X)=R(X)L(X,\rho,\rho^\vee)$ must belong to 
$\Fl[X^{\pm 1}]$. By our induction hypothesis, we have
\[L(X,St(\nutwo\rho,k_1-1),\rho^\vee)=L(X,\nutwo^{k_1-1}\rho,\rho^\vee),\] and we record the equation 
\begin{equation}\label{A} L(X,St(\rho,k_1),\rho^\vee)=P(X)L(q^{1-k_1}X,\rho,\rho^\vee).\end{equation}
For the same reason, there is~$Q\in \Fl[X^{\pm 1}]$ such that 
\begin{equation}\label{A'} L(q^{-1}X^{-1},St(\rho^\vee,[1-k_1,0]),\rho)=Q(X)L(q^{-1}X^{-1},\rho^\vee,\rho).\end{equation}
Now, the relation 
\[\gamma(X,St(\rho,k_1),\rho^\vee)=\prod_{i=0}^{k_1-1} \gamma(q^{-i}X,\rho,\rho^\vee)\] implies that, up to 
units in~$\Fl[X^{\pm 1}]$, we have
\begin{equation}\label{B}\frac{L(q^{-1}X^{-1},St(\rho^\vee,[1-k_1,0]),\rho))}{L(X,St(\rho,k_1),\rho^\vee)}=
\prod_{i=0}^{k_1-1}\frac{L(q^{i-1}X^{-1},\rho^\vee,\rho)}{L(q^{-i}X,\rho,\rho^\vee)}.\end{equation}
But, up to units in~$\Fl[X^{\pm 1}]$ again, we have
\[L(q^{i-1}X^{-1},\rho^\vee,\rho)=L(q^{1-i}X,\rho^\vee,\rho),\] hence, up to units, Equation (\ref{B}) becomes 
\begin{equation}\label{C}\frac{L(q^{-1}X^{-1},St(\rho^\vee,[1-k_1,0]),\rho))}{L(X,St(\rho,k_1),\rho^\vee)}=
\frac{L(qX,\rho^\vee,\rho)}{L(q^{1-k_1}X,\rho,\rho^\vee)}=\frac{L(q^{-1}X^{-1},\rho^\vee,\rho)}{L(q^{1-k_1}X,\rho,\rho^\vee)}.\end{equation}
 
Comparing with Equations (\ref{A}) and (\ref{A'}), we deduce that~$P=Q$ up to units, and thus that 
$P$ divides both polynomials 
\[L(q^{-1}X^{-1},\rho^\vee,\rho)^{-1}=q(\rho)^{-1}X^{-f}-1,\text{ and }L(q^{1-k_1}X,\rho^\vee,\rho)^{-1}=q(\rho)^{1-k_1}X^{f}-1.\] 
Hence~$P$ divides~$q(\rho)X^f-1$ and~$q(\rho)X^f-q(\rho)^{k_1}$ 
in~$\Fl[X^{\pm 1}]$. As~$q(\rho)^{k_1}\neq 1$, because~$k_1<e(\rho)$, the polynomials~$q(\rho)X^f-1$ and~$q(\rho)X^f-q(\rho)^{k_1}$ are coprime, hence~$P$ is a unit in~$\Fl[X^{\pm 1}]$ which must be equal to~$1$ as it is the quotient of two Euler factors, and the equality
\begin{equation}\label{D} L(X,St(\rho,k_1),\rho^\vee)=L(q^{1-k_1}X,\rho,\rho^\vee)\end{equation}
follows from Equation (\ref{A}).

Now we do the induction step, and suppose that~$k_2>1$. By Lemmas \ref{div-subquotient} and \ref{Lintermediate} again, we know that there is 
$P\in \Fl[X^{\pm 1}]$ such that 
\[L(X,\L(\D_1),\L(\D_2))=P (X) \prod_{i=0}^{k_2-1} L(X,\L(\D_1),\nutwo^i\rho^\vee),\] which by the previous case, gives the relation 
\begin{equation}\label{a}  L(X,\L(\D_1),\L(\D_2))=P (X) \prod_{i=0}^{k_2-1} L(X,\nutwo^{k_1-1}\rho,\nutwo^i\rho^\vee)\end{equation}
By the same argument, we obtain the existence of~$Q\in\Fl[X^{\pm 1}]$ such that 
\begin{equation}\label{a'}  L(q^{-1}X^{-1},\L(\D_1)^\vee,\L(\D_2)^\vee)=Q (X) \prod_{i=0}^{k_2-1} 
L(q^{-1}X^{-1},\rho^\vee,\nutwo^{-i}\rho).\end{equation}
Up to units in~$\Fl[X^{\pm 1}]$, we have:
 \[\gamma(X,\L(\D_1),\L(\D_2))=\prod_{i=0}^{k_2-1}\gamma(X,\L(\D_1),\nutwo^i\rho^\vee)=
\prod_{i=0}^{k_2-1}\frac{L(q^{-1}X^{-1},\L(\D_1)^\vee,\nutwo^{-i}\rho)}{L(X,\L(\D_1),\nutwo^i\rho^\vee)}\]
\begin{equation}\label{b}=\prod_{i=0}^{k_2-1}
\frac{L(q^{-1}X^{-1},\rho^\vee,\nutwo^{-i}\rho)}{L(X,\nutwo^{k_1-1}\rho,\nutwo^i\rho^\vee)}.\end{equation}
By Equations (\ref{a}) and (\ref{a'}), we deduce that~$P=Q$ up to units in~$\Fl[X^{\pm 1}]$, and in fact~$P=Q$ as both must belong to 
$\Fl[X]$ with constant term~$1$. On the other hand, by Lemma \ref{important-cor}, we know that 
$L(X,\L(\D_1),St(\rho,k_2-1))$ divides~$L(X,\L(\D_1),\L(\D_2))$, and 
\[L(X,\L(\D_1),St(\rho,[1,k_2-1]))=\prod_{i=1}^{k_2-1} L(X,\nutwo^{k_1-1}\rho,\nutwo^i\rho^\vee)\] by our induction hypothesis. This implies, 
by Equation (\ref{a}), that~$P$ divides~$L(X,\nutwo^{k_1-1}\rho,\rho^\vee)^{-1}$. Similarly, we find that~$P$ 
divides~$L(q^{-1}X^{-1},\rho^\vee,\rho)^{-1}$. As we have already proved in the case~$k_2=1$ that 
$L(X,\nutwo^{k_1-1}\rho,\rho^\vee)^{-1}$ and~$L(q^{-1}X^{-1},\rho^\vee,\rho)^{-1}$ are coprime, we deduce that~$P=1$, and Equation (\ref{a}) gives the desired equality.
\end{proof}

\subsection{$L$-factors of banal generic representations}

We recall the following fact. If $\rho$ is an $\ell$-adic cuspidal representation of $G_n$, with $f=f(\rho)$, then 
by \cite[6.2.5]{BK93}, the integer $f$ is th order of the cyclic group of unramified characters of $F^*$ fixing $\rho$.

\begin{prop}\label{common pole implies equal}
Let~$R$ be~$\Ql$ or~$\Fl$, and let~$\rho$,~$\rho'$ and~$\rho''$ be cuspidal~$R$-representations of~$G_n$, such that~$L(X,\rho,\rho')$ and 
$L(X,\rho,\rho'')$ have a common pole, then~$\rho'\simeq \rho''$. 
\end{prop}
\begin{proof}
First, this implies that there are unramified characters~$\chi'$ and~$\chi''$ of~$F^\times$ such that 
$\rho'\simeq \chi'\rho^\vee$ and~$\rho \simeq \chi''\rho^\vee$, so~$f(\rho)=f(\rho')=f(\rho'')=f$. Now 
if~$x$ is this common pole, which is necessarily in~$R^\times$, we have~$1=x^f\chi''(\w)^f=x^f\chi'(\w)^f$, i.e.~$\chi''(\w)^f=\chi'(\w)^f$. 
Calling~$\mu$ the character~$\chi''\chi'^{-1}$, this implies that~$\mu^f=\1$, hence~$\mu\rho'\simeq\rho'$ and the result follows. 
\end{proof}

\begin{cor}\label{cor-distinct-cusp-lines}
Let~$\rho$ be a cuspidal~$R$-representation of~$G_m$, and~$\pi_1$ and~$\pi_2$ two~$R$-representations of Whittaker type supported 
on two different cuspidal lines~$\Z_{\rho_1}$ and~$\Z_{\rho_2}$. Then~$L(X,\rho,\pi_1)$ and 
$L(X,\rho,\pi_2)$ have no common pole.
\end{cor}
\begin{proof}
It follows at once from Proposition \ref{common pole implies equal} and Lemma \ref{Lintermediate}.
\end{proof}

\begin{cor}\label{cor-distinct-cusp-lines-2}
Let~$\rho$ be a cuspidal~$R$-representation of~$G_m$, and~$\pi_1$ and~$\pi_2$ two~$R$-representations of Whittaker type supported 
on two different cuspidal lines~$\Z_{\rho_1}$ and~$\Z_{\rho_2}$. Then~$L(X,\rho,\pi_1)$ and 
$L(q^{-1}X^{-1},\rho^\vee,\pi_2^\vee)$ have no common pole.
\end{cor}
\begin{proof}
If~$L(q^{-1}X^{-1},\rho^\vee,\pi_2^\vee)$ has a pole, by Lemma \ref{Lintermediate}, it occurs as a pole of 
$L(q^{-1}X^{-1},\rho^\vee,\nu^{k}\rho_2^\vee)$ for some~$k\in \Z$. But, up to a unit of~$R[X^{\pm 1}]$, the rational map
~$L(q^{-1}X^{-1},\rho^\vee,\nu^{k}\rho_2^\vee)$ is equal to~$L(qX,\rho,\nu^{-k}\rho_2)=L(X,\rho,\nu^{-1-k}\rho_2)$. 
The result now follows from Corollary \ref{cor-distinct-cusp-lines}.
\end{proof}

\begin{LM}\label{no-common-pole}
Let~$\rho$ be a cuspidal~$\ell$-modular representation of~$G_m$, and let 
$$\pi=\L(\D_1)\times \dots \times \L(\D_r)$$ be a banal generic~$\ell$-modular representation 
of~$G_n$. If~$\D_1=[a_1,b_1]_{\rho_1}$ and 
$\D_2=[a_2,b_2]_{\rho_2}$, then if~$0 \leqslant k\leqslant \min(b_1-a_1,b_2-a_2)$, the factors~$L(X,\rho,\L(\D_1))$ and 
$L(q^{-1}X^{-1},\nutwo^{-k}\rho^\vee,\L(\D_2)^\vee)$ have no common pole.
\end{LM}
\begin{proof}
Suppose there is a common pole~$x$. First by Corollary \ref{cor-distinct-cusp-lines-2},~$\rho_2$ and~$\rho_1$ are on the same cuspidal line, 
in particular we can take~$\rho_1=\rho_2=\tau$. If~$\tau$ is not equal to~$\rho^\vee$ up to twisting by an unramified character or~$\rho$ is non-banal, then both factors would be equal to~$1$ and 
we are done. Hence~$\rho$ is banal,~$\tau=\chi\rho^\vee$ for~$\chi$ an unramified character of~$F^\times$, and we set~$f=f(\rho)$.  Now 
$$L(X,\rho,\L(\D_1))=1/(1-(q^{-b_1}\chi(\w)X)^f)$$ and~$$L(q^{-1}X^{-1},\rho^\vee,\nutwo^{-k}\L(\D_2)^\vee)=1/(1-(q^{a_2+k-1}\chi(\w)^{-1}X^{-1})^f),$$ setting~$y=\chi(\w)x$, 
we get the relation~$y^f=q^{f(a_2+k-1)}=q^{b_1f}$, i.e. 
\begin{equation}\label{banal-equality} q(\rho)^{a_2}=q(\rho)^{b_1+1-k}.\end{equation}
 As~$\pi$ is banal, we can always choose~$a_1\leqslant b_1$ 
and~$a_2\leqslant b_2$ such that, if~$c=\max(b_1,b_2)$: either~$a_1\leqslant a_2$ and~$0\leqslant c-a_1<e-1$, or~$a_2\leqslant a_1$ and~$0\leqslant c-a_2<e-1$. In the first case, 
Equation (\ref{banal-equality}) implies~$q(\rho)^{a_2-a_1}=q(\rho)^{b_1-a_1+1-k}$, but as~$0\leqslant a_2-a_1< e-1$ and~$$1\leqslant b_1-a_1+1-k\leqslant b_1-a_1+1<e,$$ we deduce that~$a_2-a_1=b_1-a_1+1-k$, i.e.~$a_2=b_1+1-k$. In this case 
one would have~$a_1\leqslant a_2\leqslant b_1+1$ and~$b_2=b_1+1-k+b_2-a_2\geqslant b_1+1$, which is absurd 
as~$\D_1$ and~$\D_2$ would then be linked. In the second case, 
Equation (\ref{banal-equality}) implies~$q(\rho)^{1-k+b_1-a_2}=1$. 
As~$1\leqslant 1+[(b_1-a_1)-k]\leqslant 1+b_1-a_2-k \leqslant 1+b_1-a_2< e$, we deduce the equality~$b_1-a_2= k-1$, which is absurd as 
$k\leqslant b_1-a_1\leqslant b_1-a_2$. This ends the proof.
\end{proof}

We now prove the inductivity relation for banal generic representations.

\begin{thm}\label{inductivity-banal-generic}
Let~$\pi=\L(\D_1)\times \dots \times \L(\D_r)$ and 
$\pi'=\L(\D'_1)\times \dots \times \L(\D'_{r'})$ be two generic~$\ell$-modular representations of~$G_n$ and~$G_{n'}$ respectively, then 
\[L(X,\pi, \pi')=\prod_{i,j} L(X,\L(\D_i),\L(\D'_j)).\]
\end{thm}
\begin{proof}
We proceed by induction on~$n+n'$.  If~$n+n'=0$, then both sides of the equality are equal to~$1$.
We now do the induction step. Thanks to Lemma \ref{Lintermediate}, there is 
$P$ and~$Q$ in~$R[X]$, with~$P(0)=Q(0)=1$, such that 
\begin{equation}\label{ind-gen-1} L(X,\pi, \pi')=P\prod_{i,j} L(X,\L(\D_i),\L(\D'_j)),\end{equation} and 
\begin{equation}\label{ind-gen-2} L(q^{-1}X^{-1},\pi^\vee, \pi'^\vee)=Q\prod_{i,j} 
L(q^{-1}X^{-1},\L(\D_i)^\vee,\L(\D_j)^\vee).\end{equation}
The inductivity relation of~$\gamma$-factors (Theorem \ref{inductivity-gamma}) gives, up to units of~$R[X^{\pm 1}]$, the relation:
\begin{equation}\label{ind-gen-3} \frac{L(q^{-1}X^{-1},\pi^\vee, \pi'^\vee)}{L(X,\pi,\pi')}= 
\frac{\prod_{i,j} L(q^{-1}X^{-1},\L(\D_i)^\vee,\L(\D'_j)^\vee)}{\prod_{k,l} L(X,\L(\D_k),\L(\D'_l))}.
\end{equation}
In particular, Equations (\ref{ind-gen-1}), (\ref{ind-gen-2}) and (\ref{ind-gen-3}) give~$P=Q$ up to units of~$R[X^{\pm 1}]$, i.e.~$P=Q$ (as they are in~$R[X]$ with constant term equal to~$1$). It thus suffices to prove that they are equal to~$1$, which will follow from the fact that they have no common root.
Notice that if~$r=r'=1$ there is nothing to prove and~$P$ and~$Q$ 
are obviously equal to~$1$. If not, we can always assume that~$r\geqslant 2$ as the~$L$-function is by definition symmetric with respect to 
its two last variables. So we suppose that~$r\geqslant 2$. In this case, we order the segments~$\D_1,\dots,\D_r$ and~$\D'_1,\dots,\D'_{r'}$ so that~$\D_1$ has minimal length amongst the~$\D_i$  and, similarly,~$\D'_1$ has minimal length amongst the~$\D'_i$ , which is possible as~$\pi$ and~$\pi'$ are generic. We have~$\L(\D_1)=St(\rho,k)$, where~$\rho$ is a cuspidal representation of~$G_l$. We set~$\L(\D_1)^{(l)}=St(\rho,[1,k-1])$ (this is indeed the~$l$-th derivative of~$\L(\D_1)$ by Theorem \ref{derivatives-commute}, as the same formula holds for~$\ell$-adic representations according to \cite{Z}, we won't use this fact anyway). Thanks to our ordering, the representation~$\L(\D_1)^{(l)}\times \L(\D_2) \times \cdots\times \L(\D_r)$ is still generic. By induction hypothesis, we have 
\begin{align}\label{ind-gen-4}L(X,\L(\D_1)^{(l)}\times \L(\D_2)\times \cdots \times \L(\D_r),&\L(\D_1')\times \cdots \times \L(\D'_{r'}))\\
\notag&=\prod_{i=1}^{r'} L(X,\L(\D_1)^{(l)},\L(\D'_i))\prod_{j=2}^r\prod_{k=1}^{r'} L(X,\L(\D_j),\L(\D'_k)).\end{align} 
Now by Corollary \ref{important-cor}, 
we also know that~$L(X,\L(\D_1)^{(l)}\times \L(\D_2)\times \dots \times \L(\D_r),\L(\D_1')\times \dots \times \L(\D'_{r'}))^{-1}$ 
divides~$L(X,\pi, \pi')^{-1}$. Hence Equations (\ref{ind-gen-1}) and (\ref{ind-gen-4}) together imply that 
$P$ divides \[\prod_{i=1}^{r'} \frac{L(X,\L(\D_1)^{(l)},\L(\D'_i))}{L(X,\L(\D_1),\L(\D'_i)}.\] 
Write~$\L(\D'_i)=St(\rho'_i,k_i)$, and let~$\epsilon_i\in \{0,1\}$, be equal to~$1$ if and only if 
$\rho'_i$ is equal to~$\rho^\vee$ up to unramified twist and~$k<k_i$. The above product, by Theorem \ref{inductivity-segments}, 
is equal to \[\prod_{i=1}^{r'} L(X,\rho,\L(\D'_i))^{-\epsilon_i}.\] Similarly, we show that~$Q$ divides 
\[ \prod_{i=1}^{r'} L(q^{-1}X^{-1},\nu^{-k}\rho,\L(\D'_i))^{-\epsilon_i}.\]
In particular, if~$P$ and~$Q$ have a common root, there would be~$i$ and~$j$ in~$\{ 1,\ldots,r'\}$, with~$k_i$ and~$k_j$ both~$>k$, such that~$L(X,\rho,\L(\D'_i))$ and~$L(q^{-1}X^{-1},\nu^{-k}\rho,\L(\D'_j))$ would have a common pole~$x$. This is absurd according to Lemma \ref{no-common-pole}. This ends the proof.
\end{proof}

We obtain the following corollary. 

\begin{thm}
Let~$\pi$ and~$\pi$ be two~$\ell$-modular banal generic representations of~$G_n$, then~$\pi$ and~$\pi'$ 
admit generic lifts~$\tau$ and~$\tau'$, such that 
\[L(X,\pi,\pi')=r_\ell (L(X,\tau,\tau')).\] 
\end{thm}
\begin{proof}
Take~$\tau$ and~$\tau'$ two standard lifts which are generic, and apply Theorems 
\ref{inductivity-banal-generic} and \ref{inductivity-segments}.
\end{proof}

\subsection{$L$-factors of generic representations}\label{Finalsubsect}

We now obtain the inductivity relation for all generic representations. It follows at once from 
the results of the previous subsection and the following equality.

\begin{thm}\label{banal-part-factors}
Let~$\pi=\pi_{b}\times \pi_{tnb}$ be a generic~$\ell$-modular representation of~$G_n$, and~$\pi'=\pi_{b}'\times \pi_{tnb}'$ be a generic~$\ell$-modular representation of 
$G_{n'}$, then we have
\[L(X,\pi,\pi')= L(X,\pi_b,\pi_b').\]
\end{thm}
\begin{proof}
From Proposition \ref{div-induced}, we know that~$L(X,\pi_b,\pi_b')$ divides~$L(X,\pi,\pi')$. Now from by Lemma \ref{Lintermediate} 
we know that~$L(X,\pi,\pi')$ divides~$L(X,\pi_b,\pi_b')L(X,\pi_{tnb},\pi_b')L(X,\pi_b,\pi_{tnb}')L(X,\pi_{tnb},\pi_{tnb}')$. Hence we only need 
to see that each factor~$L(X,\pi_{tnb},\pi_b')$,~$L(X,\pi_b,\pi_{tnb}')$ and~$L(X,\pi_{tnb},\pi_{tnb}')$ is equal to one. 
This follows easily from Lemma \ref{Lintermediate} and Theorem \ref{inductivity-segments}.
\end{proof}

An immediate corollary, which follows from the inductivity relation for~$L$ factors of~$\ell$-modular 
banal generic representations (Theorem \ref{inductivity-banal-generic}), is the inductivity relation for~$L$ factors of~$\ell$-modular generic representations.

\begin{cor}\label{generic-inductivity}
Let~$\pi=\L(\D_1)\times \dots \times \L(\D_r)$ and~$\pi'=\L(\D_1')\times \dots \times \L(\D_t')$ be two~$\ell$-modular 
generic representations, then \[L(X,\pi,\pi')=\prod_{i,j}L(X,\L(\D_i),\L(\D_j')).\]
\end{cor}

As a consequence, we get a nice result on congruences of~$\ell$-adic~$L$-factors. 
We start with the particular case of segments. 

\begin{prop}\label{gcd-segments}
Let~$\L(\D)$ and~$\L(\D')$ be two~$\ell$-modular generic segments of~$G_n$ and~$G_m$ respectively. 
Then there are~$\ell$-adic generic representations~$\tau_1,\dots,\tau_p$, and~$\tau_1',\dots,\tau_q'$ (with~$p$ and~$q$ at most~$2$) such that~$J_\ell(\tau_i)=\L(\D)$ and~$J_\ell(\tau_i)=\L(\D')$, such that 
\[L(X,\L(\D),\L(\D'))=\GCD_{i,j}(L(X,\tau_i,\tau_j')).\]
\end{prop}
\begin{proof}
We write~$\D=[a,b]_\rho$ and~$\D'=[a',b']_\rho'$. In this case we already know the 
result if both segments are banal by Theorem \ref{inductivity-segments} (one only needs one lift). If one segment 
is nonbanal, for instance~$\D$, then~$\rho$ is nonbanal as well (Proposition \ref{banal-part}). We let~$\L(D')$ be a segment lifting 
$\L(\D)$ in any case. If~$\rho$ is supercuspidal, according to Proposition \ref{supercuspliftingnoniso},~$\rho$ has two cuspidal 
lifts~$\sigma_1$ and~$\sigma_2$ which are in different orbits under unramified twists. In particular, as the~$L$ factor of a pair 
of cuspidal representations which are not contragredient of one another up to unramified twist is equal to~$1$, using Lemmas \ref{div-subquotient} and \ref{Lintermediate} 
for~$L(X,L([a,b]_{\sigma_1}),\L(D'))$ and~$L(X,L([a,b]_{\sigma_2}),\L(D'))$, this implies that one of the factors is equal to~$1$, and we are done in this case (with~$\tau_i=\L([a,b]_{\sigma_i})$, and~$\tau_1'=\L(D')$). If~$\rho$ is not supercuspidal, we can write under the form~$St_r(\mu)$ for some supercuspidal representation~$\mu$. Let~$t$ 
be the positive integer such that~$\rho$ is a representation of~$G_t$, then~$\mu$ is a representation of~$G_d$ for~$d\leq t/\ell < t$, 
the important point here being~$d<t$. 
Then~$\rho$ can be lifted to a cuspidal representation~$\sigma_1$ of~$G_l$, but we also have~$\rho=J_\ell(St_r(\nu))$, for~$\nu$ a cuspidal 
lift of~$\mu$. We let~$\tau_1=\L([a,b]_{\sigma_1})$, and~$\tau_2$ be the generic subquotient of~$\nu^a St_r(\nu)\times \dots \times \nu^b St_r(\nu)$, then Lemma \ref{Jl} 
implies that~$\D=J_\ell(\tau_2)=J_\ell(\tau_1)$. 
As the~$L$ factor of a pair 
of cuspidal representations which are representations of different linear groups is equal to~$1$, by Lemmas \ref{div-subquotient} and \ref{Lintermediate} again, one of the factors~$L(X,\tau_1,\L(D'))$ and~$L(X,\tau_2,\L(D'))$ is equal to~$1$ and this ends the proof (with~$\tau_1'=\L(D')$).
\end{proof}

\begin{thm}\label{GCDtheorem}
Let~$\pi$ and~$\pi'$ be two generic~$\ell$-modular representations of~$G_n$ and~$G_m$, then:
\[L(X,\pi,\pi')=\GCD(r_\ell(L(X,\tau,\tau'))).\]
where the divisor is over all integral generic~$\ell$-adic representations~$\tau$ of~$G_n$ and~$\tau'$ of~$G_m$ which contain~$\pi$ and~$\pi'$, respectively, as subquotients after reduction modulo~$\ell$.
\end{thm}
\begin{proof}
We use notations of Lemma \ref{Jl}. Let~$\tau$ and~$\tau'$ be two~$\ell$-adic generic representations of~$G_n$, and 
let~$\pi=J_\ell(\tau)$ and~$\pi'=J_\ell(\tau')$. We have a surjection~$p: \sigma \twoheadrightarrow \tau$ (resp.~$p':\sigma'\twoheadrightarrow \tau'~$), where~$\sigma=r_1\times \dots \times r_t$ 
(resp.~$r_1'\times \dots \times r_s'$) is a product of cuspidal representations, hence a representation of Whittaker type. By Lemma \ref{div-subquotient}, we have~$L(X,\tau,\tau')=L(X,\sigma,\sigma')$. We set~$\rho_i=r_\ell(r_i)$ and~$\rho_i'=r_\ell(r_i')$, so that Lemma 
\ref{Jl} tells us that the representation~$\pi$ (resp.~$\pi'$) is the unique generic subquotient of 
$\rho_1\times \dots \times \rho_t$ (resp.~$\rho_1' \times \dots \times \rho_s'$).
By Lemma \ref{div-subquotient} again, the factor~$L(X,\pi,\pi')$ divides 
\[L(X,\rho_1\times \dots \times \rho_t,\rho_1' \times \dots \times \rho_s'),\] which in turn divides 
$r_\ell(L(X,\sigma,\sigma'))=r_\ell(L(X,\tau,\tau'))$ according to Theorem \ref{compat1}. This proves that~$L(X,\pi,\pi')$ divides the greatest common divisor considered in the statement.

We now have to show that this greatest common divisor divides~$L(X,\pi,\pi')$. We will thus find a finite number of~$\ell$-adic generic representations 
$\tau_1,\dots,\tau_u$ and~$\tau_1',\dots,\tau_v'$, with~$\pi=J_\ell(\tau_i)$ for all~$i$ and~$\pi'=J_\ell(\tau_j')$ for all 
$j$, such that~$\GCD_{i,j}(r_\ell(L(X,\tau_i,\tau_j')))$ divides~$L(X,\pi,\pi')$. We already know that this is true when~$\pi$ and 
$\pi'$ are generic segments according to Proposition \ref{gcd-segments}. In general, we write 
\[\pi=\L(\D_{1})\times \dots \times \L(\D_r),\quad \pi'=\L(\D_1')\times \dots \times \L(\D_s').\]
According to Proposition \ref{gcd-segments}, for each~$k$ between~$1$ and~$r$ (resp. each~$l$ between~$1$ and~$s$) we can choose a finite number of generic representations~$\tau_{i_k}^k$ (resp.~$\tau_{j_l}^l$) such that~$J_\ell(\tau_{i_k}^k)=\L(\D_k)$ (resp.~$J_\ell({\tau'}_{j_l}^l)=\L(\D_l')$), and 
\[L(X,\L(\D_k),\L(\D_l'))=\GCD_{i_k,j_l}(r_\ell(L(X,\tau_{i_k}^k,{\tau'}_{j_l}^l))).\] 
We denote by~$\tau_{i_1,\dots,i_r}$ the unique generic subquotient of~$\tau_{i_1}^1\times \dots \times \tau_{i_r}^r$, and~${\tau'}_{j_1,\dots,j_s}$ the unique generic subquotient of~${\tau'}_{j_1}^1\times \dots \times {\tau'}_{j_s}^s$, so that~$J_\ell(\tau_{i_1,\dots,i_r})=\pi$, and~$J_\ell(\tau_{i_1,\dots,i_s}')=\pi'$. By Lemmas \ref{div-subquotient} and 
\ref{Lintermediate}, we know that~$L(X,\tau_{i_1,\dots,i_r},\tau_{j_1,\dots,j_s}')$ divides the factor 
$\prod_{k=1,l=1}^{r,s} L(X,\tau_{i_k}^k,{\tau'}_{i_l}^l).$ Hence 
\[\underset{(i_1,\dots,i_r),(j_1,\dots,j_s)}{\GCD}\left(L(X,\tau_{i_1,\dots,i_r},\tau_{j_1,\dots,j_s}')\right)\mid \underset{(i_1,\dots,i_r),(j_1,\dots,j_s)}{\GCD}\left(\prod_{k=1,l=1}^{r,s} r_\ell(L(X,\tau_{i_k}^k,{\tau'}_{i_l}^l))\right).\]
This latter GCD is equal to
\[\prod_{k=1,l=1}^{r,s} \GCD_{i_k,j_l}\left(r_\ell(L(X,\tau_{i_k}^k,{\tau'}_{i_l}^l))\right)=
\prod_{k=1,l=1}^{r,s} L(X,\L(\D_k),\L(\D_l'))=L(X,\pi,\pi'),\]
the last equality according to Corollary \ref{generic-inductivity}. This ends the proof.
\end{proof}                                                                   
\begin{rem}
Taking the unique generic subquotient in the reduction modulo~$\ell$ of an integral generic~$\ell$-adic representation is an instance of Vign\'eras'~$J_{\ell}$ map (cf. \cite[1.8.4]{Viginv}).  As this map is defined on all integral irreducible~$\ell$-adic representations of~$G_n$, and is surjective onto all irreducible~$\ell$-modular representations of~$G_n$ [ibid.], it is tempting to define the~$L$-factor~$L(X,\pi,\pi')$, of pairs of irreducible (not necessarily generic)~$\ell$-modular representations~$\pi,\pi'$, by a~$\GCD$ of the reductions modulo~$\ell$ of the~$\ell$-adic~$L$-factors~$L(X,\tau,\tau')$ where~$\tau,\tau'$ vary over all integral~$\ell$-adic representations such that~$J_{\ell}(\tau)=\pi$ and~$J_{\ell}(\tau')=\pi'$.  Interesting future questions revolve around determining whether this is a natural definition, perhaps, by finding an integral representation for this~$\ell$-modular~$L$-factor.
\end{rem}

\bibliographystyle{plain}
\bibliography{Modlfactors}

\end{document}